\documentclass[12pt]{amsart}
\usepackage[english]{babel}
\usepackage{graphicx}
\usepackage{amsmath,amssymb,hyperref,srcltx}
\usepackage{color}
\usepackage{amsthm}
\newcommand{\diff}[2]{\mbox{{\rm Diff}{${\,}_{#1}({\mathbb C}^{#2},0)$}}}
\newcommand{\diffh}[2]{\mbox{$\widehat{\rm Diff}{{\,}_{#1}({\mathbb C}^{#2},0)}$}}

\newcommand{\cn}[1]{\mbox{(${\mathbb C}^{#1},0$)}}

\newtheorem{pro}{Proposition}[section]
\newtheorem{teo}{Theorem}[section]

\newtheorem{cor}{Corollary}[section]

\newtheorem{lem}{Lemma}[section]

\theoremstyle{remark}
\newtheorem{rem}{Remark}[section]

\theoremstyle{definition}
\newtheorem{defi}{Definition}[section]

\theoremstyle{plain}
\newtheorem{teol}{Theorem}

\begin{document}

\title[]
{Recurrent orbits of subgroups of local complex analytic diffeomorphisms}

\author{Javier Rib\'{o}n}
\address{Instituto de Matem\'{a}tica, UFF, Rua M\'{a}rio Santos Braga S/N
Valonguinho, Niter\'{o}i, Rio de Janeiro, Brasil 24020-140}
\thanks{e-mail address: javier@mat.uff.br}
\thanks{MSC class. Primary: 37F75, 20F16; Secondary: 20F14, 32H50}
\maketitle

\bibliographystyle{plain}

\section*{Abstract}
We show recurrent phenomena for orbits of groups of local complex analytic diffeomorphisms
that have a certain subgroup or image by a morphism of groups that is non-virtually solvable.
In particular we prove that a non-virtually solvable subgroup of local biholomorphisms has
always recurrent orbits, i.e. there exists an orbit contained in its set of limit points.
\section{Introduction}
We are interested in the interaction between algebraic properties of
groups of local complex analytic diffeomorphisms and their topological dynamics.
We denote by $\diff{}{n}$ the group of local complex analytic diffeomorphisms
defined in a neighborhood of the origin of ${\mathbb C}^{n}$.
The algebraic nature of subgroups of $\diff{}{n}$ is studied from different points of view
in the literature,
for instance in the context of groups of real analytic diffeomorphisms in compact manifolds
\cite[Ghys]{Ghys-identite}, the existence of faithful analytic actions of mapping
class groups of surfaces on surfaces \cite[Cantat-Cerveau]{Cantat-Cerveau},
the study of integrability properties of one-dimensional foliations
\cite[Rebelo-Reis]{RR:note} \cite[C\^{a}mara-Scardua]{Ca-Sca:finite},
local intersection dynamics \cite[Seigal-Yakovenko]{Seigal-Yakovenko:ldi}
\cite[Binyamini]{Binyamini:finite},
the study of the derived length \cite[Martelo-Rib\'{o}n]{JR:arxivdl}
\cite{JR:arxivdl2}...

This paper generalizes two results of Rebelo and Reis \cite{RR:arxiv}
(Theorems \ref{teo:rec} and \ref{teo:lcoivs})
that relate algebraic properties of a subgroup $G$ of $\diff{}{2}$
(being non-virtually solvable) with the existence of non-discrete orbits
for the action of $G$ or more rigorously for the action of any  pseudogroup
that is a representative of $G$.
The biggest generalization is that our results hold for every dimension.

Let us introduce the main results in the paper. Theorem \ref{teo:rec} provides 
a series of sufficient conditions for the non-discreteness of a subgroup of 
$\diff{}{n}$.
\begin{teol}
\label{teo:rec}
Let $G$ be a subgroup of $\mathrm{Diff} ({\mathbb C}^{n},0)$
such that at least one of the following conditions holds:
\begin{enumerate}
\item $G$ is a non-solvable unipotent subgroup of
$\mathrm{Diff} ({\mathbb C}^{n},0)$.
\item $G$ is a non-virtually solvable subgroup of $\mathrm{Diff} ({\mathbb C}^{n},0)$
consisting of non-hyperbolic elements.
\item $\overline{(j^{1} G)}_{0}$ is a non-solvable subgroup of $\mathrm{GL}(n,{\mathbb C})$.
\item $j^{1} G$ is not virtually reducible and $\overline{(j^{1} G)}_{0} \not \subset {\mathbb C}^{*} Id$.
\item $j^{1} G$ is not virtually reducible and the group induced by
$\overline{j^{1} G}$ in $\mathrm{PGL}(n,{\mathbb C})$ is non-discrete.
\end{enumerate}
Then there exists a finite subset ${\mathcal S}$ of $G$ such that
for every pseudogroup ${\mathcal P}$ that is a representative of $\langle {\mathcal S} \rangle$
there exist a connected open neighborhood $V$ of $0$  and a sequence
$(f_{n}:V \to f_{n}(V))_{n \geq 1}$ in ${\mathcal P} \setminus \{ Id_{V} \}$ converging uniformly to
$Id$ in $V$.
In particular all the points in
$V$ except at most a countable union of proper analytic sets are
recurrent for the action of ${\mathcal P}$.
\end{teol}
Let us explain the previous theorem.
We say that a local diffeomorphism $\phi \in \diff{}{n}$ is {\it unipotent} if its linear part
$D_{0} \phi$ at $0$ is a unipotent linear automorphism of ${\mathbb C}^{n}$.
We say that a subgroup $G$ of $\diff{}{n}$ is unipotent if all its elements are unipotent.
A unipotent group of local diffeomorphisms
is ``well-represented" by a Lie algebra of nilpotent formal vector fields.
As a consequence problems in unipotent subgroups of $\diff{}{n}$ can 
be interpreted as simpler problems in their Lie algebras. 

Condition (2) requires that $G$ has no finite index solvable subgroup.
Moreover the definition of hyperbolic local diffeomorphism is not the most
common one, we say that $\phi \in \diff{}{n}$ is hyperbolic if 
$\mathrm{spec} (D_{0} \phi)$ is not contained in ${\mathbb S}^{1}$.
Regarding Condition (3), $\overline{j^{1} G}$ is the closure of
$j^{1} G := \{ D_{0} \phi : \phi \in G \}$
in $\mathrm{GL}(n,{\mathbb C})$
in the usual topology. It is well-known that $\overline{j^{1} G}$ is a Lie group
and Condition (3) is imposed in its connected component $\overline{(j^{1} G)}_{0}$
of $Id$. Conditions (4) and (5) will be discussed later on.

The thesis of the theorem is that we have a finite subset $\{g_{1},\ldots,g_{p}\}$ in
$G$ such that given any choice $\phi_{j} : U_{j} \to V_{j}$ of biholomorphism
such that the germ of $\phi_{j}$ at $0$ is equal to $g_{j}$ for $1 \leq j \leq p$,
we obtain the family $(f_{n})_{n \geq 1}$ by considering certain words on the symbols
$\{\phi_{1},\ldots,\phi_{p}, \phi_{1}^{-1},\ldots, \phi_{p}^{-1}\}$.
The biholomorphism
$f_{n}$ is defined in a domain that depends on the chosen word
but anyway such domain contains $V$ for any $n \in {\mathbb N}$.
More precisely, the words providing the elements in the sequence $(f_{n})_{n \geq 1}$
are commutators obtained by applying the so called Zassenhaus lemma that 
will be discussed later on.
This strategy was applied in dimension $2$ by Rebelo and Reis \cite{RR:arxiv}. 
We generalize it to any dimension.

Conditions (1) and (2) in Theorem \ref{teo:rec}
can be weakened  by checking them out on all subgroups of $G$.
For instance Theorem \ref{teo:rec} holds if we replace Condition (1)
with the existence of a non-solvable unipotent subgroup of $G$.

In order to introduce Theorem \ref{teo:lcoivs} 
let us define the discrete orbits property.
Consider a pseudogroup ${\mathcal P}$ of holomorphic maps, 
i.e. a family $(f_{j} : U_{j} \to V_{j})_{j \in J}$ of biholomorphisms
closed by compositions, inverses, restrictions and patching (cf. Definition \ref{def:pseudog}).
We say that ${\mathcal P}$ is a representative of a subgroup $G$ of $\diff{}{n}$ 
if $0 \in U_{j} \cap V_{j}$, $U_{j} \cup V_{j} \subset {\mathbb C}^{n}$,
$f_{j}(0)=0$ for any $j \in J$ and $G$ is the group of germs at $0$ of elements of $(f_{j})_{j \in J}$.
We say that $G$ has discrete orbits (locally discrete orbits in
\cite{RR:arxiv}) if there exists a representative pseudogroup
$(f_{j} : U_{j} \to V_{j})_{j \in J}$ of $G$  such that all orbits
${\mathcal O}(P) = \{ f_{j}(P) : P \in U_{j} \}$ are discrete sets.
\begin{teol}
\label{teo:lcoivs}
Let $G$ be a subgroup of $\mathrm{Diff}({\mathbb C}^{n},0)$ with discrete orbits.
Then $G$ is virtually solvable.
\end{teol}
Given a non-virtually solvable subgroup $G$ of $\diff{}{n}$ Theorem \ref{teo:lcoivs} provides
an orbit ${\mathcal O}$ without isolated points for every choice of a representative
pseudogroup of $G$, i.e. for every choice of domains of definition of the elements of $G$.
All points of ${\mathcal O}$ are recurrent.
Some parts of the proof are inspired in analogue results of \cite{RR:arxiv} even if there
are specific issues that are associated to the higher dimensional setting.
%
%
%
%

Theorems \ref{teo:rec} and \ref{teo:lcoivs} were known in dimension $1$ as a consequence of results of
Shcherbakov \cite{Shcherbakov(density)} and Nakai \cite{Nakai-nonsolvable}.
Indeed given a non-solvable subgroup $G$ of $\diff{}{}$
there exists a real analytic curve $\Sigma$ such that
any representative ${\mathcal P}$ of $G$ has dense orbits
in the connected components of $U \setminus \Sigma$ for some
open neighborhood $U$ of $0$ \cite{Nakai-nonsolvable}.
Moreover there exist real flows of non-trivial holomorphic vector
fields in the topological closure of any representative pseudogroup of $G$.
Theorems \ref{teo:rec} and  \ref{teo:lcoivs} are not as precise, the density condition is replaced with
the existence of recurrent orbits. 
Anyway, it is not clear how to generalize the density results to the higher dimensional setting. 
For example, consider the group $H= \{ \phi \in \diff{}{n} : f \circ \phi \equiv f \}$ for some germ 
$f$ of non-constant holomorphic function defined in the neighborhood of $0$ in 
${\mathbb C}^{n}$. 
The group is non-solvable for $n>1$ since its Lie algebra 
(the set of formal singular vector fields that have $f$ as first integral) is non-solvable \cite{JR:arxivdl}.
Analogously the subgroup of $H$ of its tangent to the identity elements is non-solvable.
Hence there are finitely generated subgroups of $H$ that are non-solvable
and consist only of tangent to the identity elements. The orbit of any point by the 
action of such a group $J$ is contained in a level set of $f$ and thus has empty interior.
In spite of this, Theorem \ref{teo:rec}(1) applies to $J$ providing recurrent points.

Every instance of Theorem \ref{teo:rec} is related to the non-solvability of the
connected component of $Id$ of a certain group. In items (1) and (2) the conditions
are equivalent to the non-solvability of $G \cap \overline{G}^{z}_{0}$. 
Let us remark that $\overline{G}^{z}$ is the Zariski-closure of $G$ (cf. section \ref{sec:proalg})
and is obtained as a projective limit of algebraic matrix groups. The group 
$\overline{G}^{z}_{0}$ is its connected component of $Id$, that is generated by 
the exponential of the Lie algebra of $\overline{G}^{z}$ (cf.  \cite{JR:arxivdl}).
The condition in item (3) applies to $j^{1} G$ and it is equivalent to the non-solvability
of $ j^{1} G \cap \overline{(j^{1} G)}_{0}$. Hence the conditions in items (1), (2) and (3)
are analogous once we consider the proper definition of Zariski-closure for subgroups
of local diffeomorphisms.
Conditions (4) and (5)
imply Condition (3) and as a consequence all the hypotheses are part of a common framework.

Let us introduce an example illustrating the previous discussion. 
Consider a subgroup $G$ of $\mathrm{SL}(2,{\mathbb C})$ that induces a non-abelian free
Kleinian subgroup of $\mathrm{PSL}(2,{\mathbb C})$ whose limit set $\Lambda$ is not
the whole Riemann sphere. We can even suppose that $G$ is free.
 Let $U$ be the dense open set consisting of 
points whose directions belong to $\hat{\mathbb C} \setminus \Lambda$.
The orbits through points in $U$ are discrete. Therefore $G$ is highly non-trivial
(it is free!) but all the connected components of $Id$ associated to subgroups 
of $G$ in items (1)-(5) of Theorem  \ref{teo:rec} are solvable.

%
%
 
 In a slightly different point of view
the paper can be interpreted as a study on how non-virtual solvability
impacts the dynamics of a subgroup of $\diff{}{n}$.
Indeed the hypotheses
in Theorem \ref{teo:rec} imply that some group canonically associated to $G$ is not virtually solvable.
This is obvious for Condition (2). Since Conditions (4) and (5) imply Condition (3), it 
suffices to show that 
we can replace non-solvable with non-virtually solvable in Conditions (1) and
(3) of Theorem \ref{teo:rec}. The group $\overline{(j^{1} G)}_{0}$ is a connected Lie group and
hence it is solvable if and only if is virtually solvable.
Analogously every virtually solvable unipotent subgroup of $\diff{}{n}$ is solvable.
 Non-virtually solvable groups are part of the dichotomy provided by the Tits alternative:
given a field $F$ of characteristic $0$
every subgroup of the group $\mathrm{GL}(n,F)$ is either virtually solvable or contains
a non-abelian free subgroup \cite{Tits}.

Unfortunately the Tits alternative does not hold a priori for subgroups of $\diff{}{n}$.
Anyway the elements of the sequence $(f_{n})_{n \geq 1}$ in Theorem \ref{teo:rec} are obtained by 
considering certain
commutators of finitely many elements whose linear parts are close to $Id$.
This process is specially simple for a non-abelian free group.
Hence the proof of Theorem \ref{teo:rec} is easier in a setting in which the Tits alternative holds.
The cases (3), (4) and (5) in Theorem \ref{teo:rec} are of Tits type, meaning
that there exists 
a non-abelian free subgroup of $j^{1} G$ on two generators whose linear
parts can be chosen arbitrarily close to $Id$ and hence we can apply
the commutator process. As a consequence we can suppose in these cases that
${\mathcal S}$ has two elements. The localization of generators of free
groups in a neighborhood of the identity profits of results of Breuillard
and Gelander on the topological Tits alternative \cite{Breuillard-Gelander:Tits}.
%
%
%
%
\subsection{Pseudo-solvable groups}
Let us explain how to deal with Condition (1) in Theorem \ref{teo:rec} where we can not use
the Tits alternative. More precisely, we can not guarantee that the hypothesis implies
the existence of a non-abelian free subgroup of $G$. We will use the notion of
pseudo-solvable subgroup of $\diff{}{n}$ as a workaround for this issue.

The definition of pseudo-solvable group was introduced by Ghys in \cite{Ghys-identite}.
It is intended to identify whether or not a group is solvable by analyzing a finite generator set
in specific cases, for instance in geometrical problems.
Ghys uses this concept to construct sequences of non-trivial real analytic diffeomorphisms in
real analytic manifolds that converge uniformly to the identity map.
He explains that the definition of pseudo-solvable is arbitrary and many others
are possible \cite{Ghys-identite}[p. 171].
We use in this paper one of the possible alternative definitions.
More precisely we introduce the definition of $p$-pseudo-solvable group
for $p \in {\mathbb N} \cup \{0\}$ generalizing Ghys definition that corresponds to
$1$-pseudo-solvable.
Our definition is intended to
take profit of the particular
algebraic structure of solvable subgroups of $\mathrm{Diff}({\mathbb C}^{n},0)$.
\begin{defi}
Let $G$ be a finitely generated group. Consider a finite set ${\mathcal S}$ of
generators of $G$. Fix $p \in {\mathbb Z}_{\geq 0}$.
By recurrence we define ${\mathcal S}_{p}(0)={\mathcal S}$ and
\[ {\mathcal S}_{p}(j+1) = \{ [f,g] \ \mathrm{or} \ [g,f] ; f \in {\mathcal S}_{p}(j), \ g \in \cup_{k= j-p}^{j} {\mathcal S}_{p}(k)
\} . \]
We say that $G$ is $p$-{\it pseudo-solvable} for ${\mathcal S}$ if there exists
$j \in {\mathbb N}$ such that ${\mathcal S}_{p}(j)=\{1\}$.
We say that $G$ is $p$-{\it pseudo-solvable} if it is $p$-pseudo-solvable for some
finite generator set. We drop the subindex $p$ when it is implicit.
\end{defi}
Let us clarify some points in the previous definition. It could happen that $j-p <0$;
in such a case we consider ${\mathcal S}(k)= \{1\}$ for any $k < 0$.
We always consider that given $f \in {\mathcal S}(0)$ its inverse $f^{-1}$ also belongs
to ${\mathcal S}(0)$. Moreover since $[f,g]^{-1} = [g,f]$ this condition holds for
${\mathcal S}(j)$ for every $j \geq 0$.

The Zassenhaus lemma implies that if ${\mathcal S}$
consists of elements close to $Id$ then the elements in ${\mathcal S}_{p}(k)$ tend uniformly
to $Id$ when $k \to \infty$. This idea was introduced by Ghys for
the study of groups of real analytic diffeomorphisms of compact manifolds
\cite{Ghys-identite} and was adapted to the local setting by Loray and Rebelo \cite{LorReb}
(cf. also \cite{RR:arxiv}).
Therefore non-$p$-pseudo-solvable groups with generators close to $Id$ provide
sequences of non-trivial elements converging to the identity map.
The condition of being close to the identity map is somehow automatic
for the elements of a unipotent subgroup of $\diff{}{n}$ and as a consequence
in order to prove Theorem \ref{teo:rec} in Case (1)
it suffices to show that $G$ is not $p$-pseudo-solvable for some $p \in {\mathbb Z}_{\geq 0}$.
We complete the proof of Theorem \ref{teo:rec} by showing that
there exists $p=p(n)$ such that a unipotent subgroup $G$ of $\diff{}{n}$
is solvable if and only if is $p$-pseudo-solvable (Theorem \ref{teo:psis}).
 Let us remark that Rebelo and Reis show 
in their proof of the analogue of Theorem \ref{teo:rec}(1) for dimension $2$
that $1$-pseudo-solvable implies solvable for $n=2$ \cite{RR:arxiv}.
We show a weaker property, easier to prove and that anyway suffices to 
prove Theorems \ref{teo:rec} and \ref{teo:lcoivs}.

The definition of pseudo-solvable is useful in a class of groups if
pseudo-solvable is equivalent to solvable; otherwise it is a property
that depends on the choice of generators of the group and with no
clear algebraic or geometrical meaning.
The main drawback of the definition of pseudo-solvable is that the proof of
this equivalence is in general quite technical.
Let us explain how working with
$p$-pseudo-solvability simplifies the proof.

Let $G$ be a unipotent  subgroup of $\diff{}{n}$, for example a group
of tangent to the identity diffeomorphisms.
We want to show that if a finitely generated subgroup of $G$ is $p$-pseudo-solvable,
for some $p = p(n)$ to be determined later on, then $G$ is solvable.
Let $m \in {\mathbb N}$ be the minimum integer such that ${\mathcal S}(m)=\{Id\}$.
We define ${\mathcal S}(j,k) = \cup_{l=j}^{k} {\mathcal S}(l)$ for $0 \leq j \leq k$,
$\Gamma (j)= \langle {\mathcal S}(j,m) \rangle$
and $G(j,j+p)=  \langle {\mathcal S}(j,j+p) \rangle$.
It is clear that $\Gamma (m)$ is solvable and the proof relies on proving that
if $\Gamma (j+1)$ is solvable then $\Gamma (j)$ is solvable for any $0 \leq j < m$.
Given $\phi \in {\mathcal S}(j)$ and $0 \leq q \leq p$ we have
$\phi \circ \psi \circ \phi^{-1} \circ \psi^{-1} \in {\mathcal S}(j+q+1)$
for any $\psi \in {\mathcal S}(j+q)$ by definition of ${\mathcal S}(j+q+1)$.
In particular we obtain
\begin{equation}
\label{equ:nor}
\phi G(j+1,j+q) \phi^{-1} \subset G(j+1,j+q+1)
\end{equation}
for $1 \leq q \leq p$. If $G(j+1,j+q) = G(j+1,j+q+1)$
then $\phi$ normalizes $G(j+1,j+q)$ and this provides valuable information about
the elements of ${\mathcal S}(j)$ that can be used to show that $\Gamma (j)$ is solvable.
Unfortunately the previous groups can be different. Anyway we associate
an invariant taking finitely many values to the the subgroups in the increasing sequence
\begin{equation}
\label{equ:inc}
G(j+1,j+1) \subset G(j+1,j+2) \subset \ldots \subset G(j+1,j+p+1) 
\end{equation}
of subgroups of the solvable group $\Gamma (j+1)$.
Moreover the invariant is increasing in a lexicographical order.
if $p$ is big enough we can always find some $1 \leq q \leq p$ such that
the invariants associated to
$G(j+1,j+q)$ and $G(j+1,j+q+1)$ coincide. The groups can be still different but
we find solvable extensions of their Lie algebras (and even solvable extensions of the groups)
that coincide and continue to satisfy the normalizing equation (\ref{equ:nor}).
The value of $p$ can be chosen as the number of different values of
the invariant. Our method associates an increasing sequence of solvable Lie algebras 
of formal vector fields to the sequence $(\ref{equ:inc})$ and then 
relies on the classification of such Lie algebras in \cite{JR:arxivdl}. 

Theorems \ref{teo:rec} and \ref{teo:lcoivs} illustrate how the description 
of the nature of an algebraic object, namely a Lie algebra of formal 
vector fields, entails dynamical consequences, and specifically recurrence 
properties, for unipotent subgroups of $\diff{}{n}$.

Resuming we show
a maximal-like condition on unipotent solvable subgroups of $\diff{}{n}$.
It allows to show   $p$-pseudo-solvable $\implies$ solvable by
turning subnormalizing group properties into normalizing properties.
\section{Preliminaries and notations}
Let us introduce some concepts that will be used in the paper.
\subsection{Groups and Lie algebras}
Let $G$ be a group. Let us define the derived groups of $G$.
\begin{defi}
We define the {\it derived group} $G^{(1)}$ (or $[G,G]$) of $G$ as the group generated by the commutators
$[f,g]:= fgf^{-1} g^{1}$ of elements of $G$. Analogously we define
$G^{(j+1)} = [G^{(j)},G^{(j)}]$ for $j \geq 1$. We denote $G^{(0)} =G$.
\end{defi}
The previous definition can be extended to any Lie algebra ${\mathfrak g}$ by replacing the commutator with
the Lie bracket.
\begin{defi}
We say that a group $G$ (resp. Lie algebra) is {\it solvable} if there exists $j \in {\mathbb N} \cup \{0\}$ such that
$G^{(j)}$ is trivial. In such a case we define its {\it derived length} $\ell (G)$ as the minimum
$j \in {\mathbb N} \cup \{0\}$ such that $G^{(j)}$ is trivial.
We say that the derived length of a non-solvable group $G$ (resp. Lie algebra) is equal to $\infty$.
\end{defi}
The statements of Theorems \ref{teo:rec}
and \ref{teo:lcoivs} involve properties of finite index subgroups of
$\diff{}{n}$ or $\mathrm{GL}(n,{\mathbb C})$.
\begin{defi}
Consider a group property $P$ (for example being abelian, free...). We say that a group $G$
is {\it virtually}  $P$ if there exists a finite index subgroup of $G$
that satisfies $P$.
\end{defi}
\begin{rem}
Every finite index subgroup $H$ of a group $G$ contains a finite index normal subgroup of $G$.
Indeed the intersection of all the conjugates of $H$ is a finite index normal subgroup of $G$.
The group properties that we consider in this paper are subgroup-closed. As a
consequence given a virtually $P$ group $G$ there exists a finite index normal subgroup of $G$
that has property $P$.
\end{rem}
\subsection{Local diffeomorphisms}
We introduce the group of linear parts of a subgroup of $\diff{}{n}$.
\begin{defi}
Let $\phi \in \diff{}{n}$. We denote by $D_{0} \phi$ (or $j^{1} \phi$) its differential at the
origin.
\end{defi}
\begin{defi}
\label{def:j1g}
Let $G$ be a subgroup of $\diff{}{n}$. We define the subgroup $j^{1} G = \{ j^{1} \phi : \phi \in G \}$
of $\mathrm{GL}(n,{\mathbb C})$.
\end{defi}
Let us define unipotent diffeomorphisms and groups. They are the objects of
Theorem \ref{teo:psis}, one of the main
results of the paper.
\begin{defi}
\label{def:unidif}
We say that a diffeomorphism $\phi \in \diff{}{n}$ is {\it unipotent} if   $D_{0} \phi$
is a unipotent isomorphism, i.e. $D_{0} \phi - Id$ is nilpotent.
We say that $\phi$ is {\it tangent to the identity} if $D_{0} \phi \equiv Id$.
\end{defi}
\begin{defi}
\label{def:unigr}
We denote by $\diff{u}{n}$ the subset of $\diff{}{n}$ consisting of unipotent diffeomorphisms.
We say that a subgroup $G$ of $\diff{}{n}$ is {\it unipotent} if $G \subset \diff{u}{n}$.
\end{defi}
\begin{defi}
We define $\diff{1}{n}$ as the subgroup of $\diff{}{n}$ consisting  of tangent to the identity
elements.
\end{defi}
\subsection{Lie groups}
Let us see that the commutator map defined in a Lie group has attracting behavior in a
neighborhood of the identity element. It will be useful to study matrix groups.
\begin{rem}[Zassenhaus lemma]
\label{rem:lie}
Let $T$ be a Lie group. Consider a left-invariant Riemannian metric defined in $T$.
The map $\theta: T \times T \to  T$ defined by $\theta (f,g) = [f,g]$ is differentiable and
$\theta_{|\{Id\} \times T} \equiv Id \equiv \theta_{T \times \{Id\}}$. Thus its differential at
$Id$ is equal to the identity map. Moreover there exists a small neighbourhood $V$ of
$Id$ in $T$ and a constant $C>0$ such that
\[ ||[f,g]- Id|| \leq C ||f-Id|| ||g -Id||  \ \ \forall f,g \in V. \]
There exists $\epsilon >0$ such that
$W:=\{ f \in T: ||f-Id|| < \epsilon \}$ is contained in $V$ and $\epsilon < 1/(2C)$.
In particular we have $||[f,g] -Id|| <\epsilon /2 $ for all $f,g \in W$.
\end{rem}
The lemma implies  that a discrete Lie group generated by elements sufficiently close to
the identity is nilpotent (cf. \cite{Raghu}[p. 147]).
\subsection{Pseudogroups}
We study the dynamics of pseudogroups induced by subgroups of $\diff{}{n}$.
First we introduce the definition of pseudogroups of homeomorphisms.
\begin{defi}
\label{def:pseudog}
Consider a family ${\mathcal F}= {\{f_{j} : U_{j} \to V_{j} \}}_{j \in J}$ of homeomorphisms where
$U_{j}, V_{j}$ are open sets of a topological space $M$ for any $j \in J$.
We say that ${\mathcal F}$ is a {\it pseudogroup} on $M$ if
\begin{itemize}
\item The identity map $Id: U \to U$ belongs to ${\mathcal F}$ for any open subset $U$ of $M$.
\item The map $f_{j}^{-1}: V_{j} \to U_{j}$ belongs to ${\mathcal F}$ for any $j \in J$.
\item Given elements $f_{j}:U_{j} \to V_{j}$ and $f_{k}:U_{k} \to V_{k}$ of ${\mathcal F}$
such that $V_{j} \cap U_{k} \neq \emptyset$ the composition
$f_{k} \circ f_{j}: f_{j}^{-1} (V_{j} \cap U_{k}) \to f_{k}(V_{j} \cap U_{k})$ belongs to ${\mathcal F}$.
\item Given an element $f: U \to V$ of ${\mathcal F}$ and a non-empty open subset $U'$ of $U$, the
restriction $f_{|U'}: U' \to f(U')$ belongs to ${\mathcal F}$.
\item Let $f:U \to V$ be a homeomorphism obtained by patching elements of ${\mathcal F}$, i.e.
there exists an open covering ${\{U_{j} \}}_{j \in J'}$ of $U$ for some subset $J'$ of $J$ such that
$f_{|U_{j}} \equiv f_{j}$ for any $j \in J'$. Then $f$ belongs to ${\mathcal F}$.
\end{itemize}
\end{defi}
\begin{rem}
In the previous definition $f_{|U_{j}}$ is considered as a map from $U_{j}$ to $f(U_{j})$
and not as a map from $U_{j}$ to $V$. Moreover if we say that an embedding
$f: U \to V$ belongs to a pseudogroup we mean that $f_{|U}:U \to f(U)$ belongs to
the pseudogroup.  We will use this conventions for the sake of simplicity.
\end{rem}
\begin{defi}
Consider a family ${\mathcal F}= {\{f_{j} : U_{j} \to V_{j} \}}_{j \in J}$ of homeomorphisms
as in Definition \ref{def:pseudog}.
We say that ${\mathcal P}$ is the pseudogroup generated by ${\mathcal F}$
if ${\mathcal P}$ is the minimal pseudogroup defined in $M$ and containing
${\mathcal F}$.
\end{defi}
Let us explain how to obtain pseudogroups induced by a subgroup $G$ of $\diff{}{n}$.
Roughly speaking we choose domains of definition for the elements of $G$
(or for the elements of a generator set).
\begin{defi}
 Consider a connected open neighborhood $U$ of $0$ in ${\mathbb C}^{n}$.
Let ${\mathcal F} = { \{ f_{j} : U_{j} \to V_{j} \}}_{j \in J}$ be a family of biholomorphisms such
that $0 \in U_{j} \cap V_{j}$, $U_{j} \cup V_{j} \subset U$, $U_{j}$ is connected
and $f_{j}(0)=0$ for any $j \in J$. We say that the pseudogroup ${\mathcal P}$
generated by ${\mathcal F}$ on $U$ is induced by a subgroup $G$ of $\diff{}{n}$ if
the group $\langle \phi_{j} : j \in J \rangle$ is equal to $G$ where $\phi_{j}$ is the
germ of $f_{j}$ at $0$.
\end{defi}
\begin{rem}
The pseudogroup induced by a subgroup of $\diff{}{n}$ is not unique.
\end{rem}
\section{The commutator property}
Consider a finitely generated subgroup $G$ of $\diff{}{n}$
that is not $p$-pseudo-solvable for ${\mathcal S}$.
In this section we see that if the elements of ${\mathcal S}$
have linear parts close to the identity map then
there exists a sequence of non-trivial maps in $G$ that converge uniformly
to $Id$ in some neighborhood of the origin. This property is crucial to
obtain recurrent orbits. Rebelo and Reis showed analogue results in the
case of groups that are not $1$-pseudo-solvable \cite{RR:arxiv}.
The generalization is fairly straightforward and we include it for the sake of
completeness.
\begin{defi}
Consider the usual norm defined in ${\mathbb C}^{n}$.
Let ${\mathbb B}_{r}^{n}$ be open the ball of center $0$ and radius $r$ in ${\mathbb C}^{n}$.
Let $f: {\mathbb B}_{r}^{n} \to {\mathbb C}^{n}$ be a holomorphic map. We define the norm of
$f$ in ${\mathbb B}_{r}^{n}$ as
\[ ||f||_{r} = \sup_{(x_{1},\ldots,x_{n}) \in {\mathbb B}_{r}^{n}} |f(x_{1},\ldots,x_{n})| . \]
\end{defi}
Notice that $||A||_{1}$ is the spectral norm for a
$n \times n$ matrix $A$.

Consider $r, \epsilon, \tau>0$ such that $4 \epsilon + \tau <r$.
Let $f,g: {\mathbb B}_{r}^{n} \hookrightarrow {\mathbb C}^{n}$ be
injective holomorphic maps such that $||f-Id||_{r} \leq \epsilon$ and
$||g-Id||_{r} \leq \epsilon$.
Then the commutator $[f,g]: {\mathbb B}_{r-4 \epsilon}^{n} \hookrightarrow {\mathbb B}_{r}^{n}$
is defined in ${\mathbb B}_{r-4 \epsilon}^{n}$ and satisfies
\begin{equation}
\label{equ:lrr}
||[f,g] - Id||_{r- 4 \epsilon - \tau} \leq \frac{2}{\tau} ||f-Id||_{r} ||g-Id||_{r} .
\end{equation}
We will use Equation (\ref{equ:lrr}) with the elements of a generator set ${\mathcal S}$
of a non-pseudo-solvable subgroup of $\mathrm{Diff}({\mathbb C}^{n},0)$.
The previous estimate was introduced by Loray and Rebelo \cite{LorReb}
(cf. also \cite{RR:arxiv}).
It is an analogue in the local setting of the Zassenhaus lemma (cf. Remark \ref{rem:lie}).

Next, let us see that if all elements of a generator set ${\mathcal S}$ of a subgroup $G$
of $\mathrm{Diff}({\mathbb C}^{n},0)$ are close enough to the identity map then
the elements of ${\mathcal S}_{p}(j)$ tend uniformly to $Id$ when $j \to \infty$.
\begin{pro}
\label{pro:lim}
Fix $p \in {\mathbb Z}_{\geq 0}$ and $\delta >0$ with $(p+2) \delta < 1/4$.
Let ${\mathcal S}$ be a subset of $\mathrm{Diff}({\mathbb C}^{n},0)$
such that ${\mathcal S}={\mathcal S}^{-1}$. Suppose
$g_{|{\mathbb B}_{1}^{n}}$ is  a well-defined injective holomorphic map
such that $||g-Id||_{1} \leq \delta/4$ for any $g \in {\mathcal S}$.
Then $f_{|{\mathbb B}_{1/2}^{n}}: {\mathbb B}_{1/2}^{n} \to {\mathbb B}_{1}^{n}$
is an injective holomorphic map
belonging to the pseudogroup generated by $\{ g_{|{\mathbb B}_{1-\delta/4}^{n}} : g \in {\mathcal S}\}$
on ${\mathbb B}_{1}^{n}$
such that $||f-Id||_{1/2} \leq \delta/2^{j+2}$ for all $j \geq 0$ and $f \in {\mathcal S}_{p}(j)$.
\end{pro}
The proof of Proposition \ref{pro:lim} is split in the next lemmas.
\begin{lem}
\label{lem:est1}
Consider the hypotheses in Proposition \ref{pro:lim}.
Let $ j \leq p+1$. Then
$f$ is defined ${\mathbb B}_{1-2j \delta}^{n}$ and satisfies
$||f-Id||_{1-2j \delta} \leq \delta / 2^{j+2}$ for any $f \in {\mathcal S}_{p}(j)$.
\end{lem}
\begin{proof}
The proposition is obvious for $j=0$. Let us show that if the result holds
for $0 \leq k \leq j < p+1$ then so it does for $j+1$.

An element of ${\mathcal S}(j+1)$ is of the form $[a,b]$ or $[b,a]$ where
$a \in {\mathcal S}(j)$ and $b \in {\mathcal S}(0,j)$.
Hence we obtain
\[ ||a-Id||_{1-2j \delta} \leq \delta / 2^{j+2} \ \mathrm{and} \
||b-Id||_{1-2j \delta} \leq \delta / 4 . \]
We consider $r=1-2j \delta$, $\epsilon=\delta/4$ and $\tau = \delta$ in Equation (\ref{equ:lrr}).
We get
\[ ||[a,b]-Id||_{1-2(j+1) \delta} \leq \frac{2}{\delta} \frac{\delta}{2^{j+2}} \frac{\delta}{4} =
\frac{\delta}{2^{j+3}} \]
and the same inequality holds for $[b,a]$.
\end{proof}
\begin{lem}
\label{lem:est2}
Consider the hypotheses in Proposition \ref{pro:lim}.
Let $j > p+1$. Then
$||f-Id||_{\kappa_{j}} \leq \delta / 2^{j+2}$ for any $f \in {\mathcal S}_{p}(j)$
where
$\kappa_{j}= 1 - \delta \left( 2(p+1) + 1 + \frac{1}{2} + \ldots + \frac{1}{2^{j-p-2}} \right)$.
\end{lem}
\begin{proof}
Let us show the result for $j=p+2$. Let $f \in {\mathcal S}(p+2)$; it is of the form
$[a,b]$ or $[b,a]$ where $a \in {\mathcal S}(p+1)$ and $b \in {\mathcal S}(1,p+1)$.
We consider $r = 1 - 2 (p+1) \delta$, $\epsilon = \delta/8$ and $\tau=\delta/2$.
We obtain
\[ ||[a,b]- Id||_{1 - (2 (p+1) +1) \delta} \leq \frac{4}{\delta} \frac{\delta}{2^{p+3}} \frac{\delta}{8} =
\frac{\delta}{2^{p+4}}  \]
and an analogous estimate holds for $[b,a]$.

Suppose that the result holds for any $p+2 \leq k \leq j$ and let us prove that so it does
for $j+1$. An element $f$ of ${\mathcal S}(j+1)$ is of the form
$[a,b]$ or $[b,a]$ where $a \in {\mathcal S}(j)$ and $b \in {\mathcal S}(j-p,j)$.
We consider $r= \kappa_{j}$, $\epsilon = \delta/2^{j-p+2}$ and $\tau = \delta/2^{j-p}$.
We have
\[ ||[a,b]- Id||_{\kappa_{j} -\delta/2^{j-p-1}} \leq \frac{2^{j+1-p}}{\delta} \frac{\delta}{2^{j+2}}
\frac{\delta}{2^{j+2-p}} = \frac{\delta}{2^{j+3}}  \]
and an analogous estimate holds for $[b,a]$. Since
$\kappa_{j+1} = \kappa_{j} - \delta/2^{j-p-1}$ we are done.
\end{proof}
\begin{proof}[Proof of Proposition \ref{pro:lim}]
Since $g({\mathbb B}_{1-\delta/4}^{n}) \subset {\mathbb B}_{1}^{n}$ for any $g \in {\mathcal S}$,
the pseudogroup generated by
$\{ g_{|{\mathbb B}_{1-\delta/4}^{n}} : g \in {\mathcal S}\}$
is induced by $\langle {\mathcal S} \rangle$ on ${\mathbb B}_{1}^{n}$.

Since $(2p+4) \delta < 1/2$, Lemmas \ref{lem:est1} and \ref{lem:est2} imply
$||f-Id||_{\frac{1}{2}} \leq \delta /2^{j+2}$ for all $j \geq 0$ and $f \in {\mathcal S}(j)$.
\end{proof}
The next propositions are consequences of Proposition \ref{pro:lim}.
The idea is that the non-$p$-pseudo-solvability of a group of diffeomorphisms $G$ can
be used to show that the induced pseudogroup is non-discrete.
\begin{pro}
\label{pro:estl}
Fix $p \in {\mathbb Z}_{\geq 0}$, $\delta >0$ with $(p+2) \delta < 1/4$.
Let ${\mathcal S}$ be a finite generator set of a subgroup $G$ of
$\mathrm{Diff}({\mathbb C}^{n},0)$.
Suppose $G$ is non-$p$-pseudo-solvable for ${\mathcal S}$.
Furthermore suppose $||D_{0} \phi - Id||_{1} \leq \delta/8$ for any $\phi \in {\mathcal S}$.
Then given any $r>0$ small enough there exists
a sequence $(f_{m})_{m \geq 1}$ in the
pseudogroup ${\mathcal P}$ generated by $\{ f_{|{\mathbb B}_{r(1-\delta/4)}^{n}} : f \in {\mathcal S} \}$
on ${\mathbb B}_{r}^{n}$
such that $f_{m}$ is defined in ${\mathbb B}_{r/2}^{n}$,
$(f_{m})_{|{\mathbb B}_{r/2}^{n}} \not \equiv Id$
for any $m \in {\mathbb N}$
and $\lim_{m \to \infty} ||f_{m} - Id||_{r/2} =0$.
In particular all the points in
${\mathbb B}_{r/2}^{n}$ except at most a countable union of proper analytic sets are
recurrent for the action of ${\mathcal P}$.
\end{pro}
\begin{proof}
Let $\upsilon_{r} (x_{1},\ldots,x_{n})=(r x_{1},\ldots, r x_{n})$ be the homothety of
ratio $r>0$. We define $\phi_{r} = \upsilon_{r}^{-1} \circ \phi \circ \upsilon_{r}$
for $\phi \in {\mathcal S}$ and $r >0$.

Let $\phi \in S$.
The map $(\phi_{r})_{|{\mathbb B}_{1}^{n}}$ is injective and
$\phi_{r} - D_{0} \phi$ is   bounded for $r>0$ small enough.
The family $(\phi_{r})_{r >0}$ satisfies $\lim_{r \to 0} ||\phi_{r} - D_{0} \phi||_{1} =0$.
We consider $r>0$ small enough such that
$||\phi_{r} - D_{0} \phi||_{1} \leq \delta / 8$ for any $\phi \in {\mathcal S}$.
We denote ${\mathcal S}' = \{ \phi_{r} : \phi \in {\mathcal S} \}$.

Since $G$ is not $p$-pseudo-solvable for ${\mathcal S}$, Proposition \ref{pro:lim} implies the existence of a sequence
$f_{j}': {\mathbb B}_{1/2}^{n} \to {\mathbb C}$
in the  pseudogroup ${\mathcal P}'$ generated by $\{ g_{|{\mathbb B}_{1-\delta/4}^{n}} : g \in {\mathcal S}' \}$
such that $f_{j}' \in {\mathcal S}_{p}'(j)$ and $0 < ||f_{j}' - Id||_{1/2} \leq \delta /2^{j+2}$
for any $j \geq 0$.
We denote $f_{j}= \upsilon_{r} \circ f_{j}' \circ \upsilon_{r}^{-1}$ for $j \geq 0$.
Then $f_{j}$ is in the pseudogroup ${\mathcal P}$, is defined in ${\mathbb B}_{r/2}^{n}$
and holds $0 < ||f_{j} - Id||_{r/2} \leq r \delta /2^{j+2}$ for any $j \geq 0$.
Moreover since ${\mathcal P}'$ is a pseudogroup on ${\mathbb B}_{1}^{n}$ by Proposition \ref{pro:lim},
hence ${\mathcal P}$ is a pseudogroup on ${\mathbb B}_{r}^{n}$.

We define $T_{j}$ as the set of fixed points of $f_{j}$ in ${\mathbb B}_{r/2}^{n}$.
We denote $S_{j} = \cap_{k \geq j} T_{k}$. A sufficient condition guaranteeing the recurrence is
$(x_{1},\ldots,x_{n}) \not \in \cup_{j \geq 0} S_{j}$. Since $S_{j}$ is an analytic set for $j \geq 0$,
the set $\cup_{j \geq 0} S_{j}$ is a countable union of proper analytic sets.
\end{proof}
\begin{pro}
\label{pro:estup}
Fix $p \in {\mathbb Z}_{\geq 0}$, $\delta >0$ with $(p+2) \delta < 1/4$.
Let $G \subset \mathrm{Diff}_{u}({\mathbb C}^{n},0)$ be a non-$p$-pseudo-solvable group.
Then given any $r>0$ small enough there exists a sequence $(f_{j})_{j \geq 1}$ in the
pseudogroup ${\mathcal P}$ generated by $\{  f_{|{\mathbb B}_{r(1-\delta/4)}^{n}} : f \in G \}$ on ${\mathbb B}_{r}^{n}$
such that $f_{j}$ is defined in ${\mathbb B}_{r/2}^{n}$,
$(f_{j})_{|{\mathbb B}_{r/2}^{n}} \not \equiv Id$ for any $j \geq 1$
and $\lim_{j \to \infty} ||f_{j} - Id||_{r/2}=0$.
In particular all the points in
${\mathbb B}_{r/2}^{n}$ except at most a countable union of proper analytic sets are
recurrent for the action of ${\mathcal P}$.
\end{pro}
We only consider diffeomorphisms $f \in G$ well-defined, injective in ${\mathbb B}_{r}^{n}$
and such that $f({\mathbb B}_{r(1- \delta/4)}^{n}) \subset {\mathbb B}_{r}^{n}$
as elements of  $\{ f_{|{\mathbb B}_{r}^{n}} : f \in G \}$.
\begin{proof}
Consider a finite generator set ${\mathcal S}$ of $G$.
Since $j^{1} G$ (cf. Definition \ref{def:j1g})
consists of unipotent matrices, it is a group of upper triangular matrices
up to a change of coordinates by Kolchin theorem (cf. \cite{Serre.Lie}[p. 35, Theorem 3*]).
In particular $j^{1} G$ is nilpotent and then solvable. There exists $m \in {\mathbb N}$
such that $(j^{1} G)^{(m)} = \{Id\}$.
The set ${\mathcal S}_{p}(j)$ is contained in $G^{(l)}$  for any $j \geq (l-1) p + l$.
We deduce that all elements of ${\mathcal S}_{p}(j)$ are tangent to the identity for
any $j \geq (m-1) p + m$.
We define ${\mathcal S}' = {\mathcal S}_{p}((m-1) p + m, m p + m)$.
We have ${\mathcal S}_{p}(m p + m) \subset {\mathcal S}_{p}'(0)$ and a simple induction argument proves
${\mathcal S}_{p}(m p + m + j) \subset {\mathcal S}_{p}'(j)$ for any $j \in {\mathbb N}$.

Since $G$ is non-$p$-pseudo-solvable for ${\mathcal S}$, the group $\langle {\mathcal S}' \rangle$ is
not $p$-pseudo-solvable for ${\mathcal S}'$. Since $\langle {\mathcal S}' \rangle$ consists of tangent
to the identity diffeomorphisms, we apply Proposition \ref{pro:estl} to $\langle {\mathcal S}' \rangle$
and ${\mathcal S}'$.
\end{proof}
It is natural that there is no condition on the generators being close to the $Id$
in the previous proposition since it is implicit. Given a group
$G \subset \mathrm{Diff}_{u}({\mathbb C}^{n},0)$ generated by a finite set ${\mathcal S}$
we can suppose that $j^{1} G$ is upper triangular by Kolchin's theorem.
Then we can suppose that $j^{1} {\mathcal S}$ is contained in an arbitrary neighborhood of
$Id$ by making a linear change of coordinates.
Thus the elements of ${\mathcal S}$ can be considered arbitrarily close to the identity map.
\section{Subgroups of unipotent elements}
In this section we show that given a subgroup $G$ of $\diff{}{n}$
(or of its formal completion) of unipotent
elements, $G$ is solvable if and only if it is $p$-pseudo-solvable for some
$p \geq p(n)$ (Theorem \ref{teo:psis}).
\subsection{Formal diffeomorphisms and vector fields}
Let us recap some facts about pro-algebraic groups of formal diffeomorphisms
in the next sections.
This material can be found mostly in \cite{JR:arxivdl} and some points are
expanded in \cite{JR:arxivdl2}.

We interpret elements of $\diff{}{n}$ as operators acting on formal power series.
Let ${\mathfrak m}$ the maximal ideal of the ring ${\mathbb C}[[x_{1},\ldots,x_{n}]]$ of formal power series.
Any $\phi \in \diff{}{n}$ induces an element $\phi_{k}$ of $\mathrm{GL}({\mathfrak m}/{\mathfrak m}^{k+1})$
by composition:
\[
\begin{array}{cclll}
\phi_{k} & : & {\mathfrak m}/{\mathfrak m}^{k+1} & \to & {\mathfrak m}/{\mathfrak m}^{k+1} \\
         &   & g + {\mathfrak m}^{k+1} & \mapsto & g \circ \phi +  {\mathfrak m}^{k+1}
\end{array}
\]
Indeed $\phi_{k}$ is an isomorphism of ${\mathbb C}$-algebras for any $k \in {\mathbb N}$.
We consider the group $D_{k} = \{ \phi_{k} : \phi \in \diff{}{n} \}$, it coincides with
the subgroup of $\mathrm{GL}({\mathfrak m}/{\mathfrak m}^{k+1})$ of isomorphisms of the
${\mathbb C}$-algebra ${\mathfrak m}/{\mathfrak m}^{k+1}$. In particular $D_{k}$ is an algebraic subgroup
of $\mathrm{GL}({\mathfrak m}/{\mathfrak m}^{k+1})$.
\begin{defi}
We define the group of formal diffeomorphisms $\diffh{}{n}$ as the projective limit
$\varprojlim_{k \in {\mathbb N}} D_{k}$.
\end{defi}
Given an element $\phi = (A_{k})_{k \geq 1}$ we have that $(A_{k}(g+ {\mathfrak m}^{k}))_{k \geq 1}$ converges in the
${\mathfrak m}$-adic topology, i.e. the Krull topology,
to an element of $\varprojlim_{k \in {\mathbb N}} {\mathfrak m}/{\mathfrak m}^{k+1}$
for any $g \in {\mathfrak m}$.
Since ${\mathfrak m}$ is equal to the inverse limit $\varprojlim_{k \in {\mathbb N}} {\mathfrak m}/{\mathfrak m}^{k+1}$,
the limit $\lim_{k \to \infty} A_{k} (g+ {\mathfrak m}^{k})$ exists in the Krull topology and belongs to ${\mathfrak m}$.
We denote
\[ \hat{\phi} = (\hat{\phi}_{1}, \ldots, \hat{\phi}_{n}) :=
(\lim_{k \to \infty} A_{k} (x_{1}+ {\mathfrak m}^{k}), \ldots, \lim_{k \to \infty} A_{k} (x_{n}+ {\mathfrak m}^{k})). \]
Since $A_{1}$ is an isomorphism, the linear map $j^{1} \hat{\phi}$ is invertible. In this way we can interpret
a formal diffeomorphism either as an element of $\varprojlim_{k \in {\mathbb N}} D_{k}$ or as
$n$-uple of elements of ${\mathfrak m}$ whose first jet is invertible.
\begin{defi}
We define unipotent formal diffeomorphisms and groups analogously as in
Definitions \ref{def:unidif} and \ref{def:unigr}. We denote by $\diffh{u}{n}$ the subset
of unipotent formal diffeomorphisms.
\end{defi}
\begin{defi}
We define $L_{k}$ as the Lie algebra of derivations of the ${\mathbb C}$-algebra
$ {\mathfrak m}/{\mathfrak m}^{k+1}$. We define the Lie algebra of (singular) formal vector fields $\hat{\mathfrak X} \cn{n}$
as $\varprojlim_{k \in {\mathbb N}} L_{k}$.
\end{defi}
Analogously as for formal diffeomorphisms, given $(B_{k})_{k \geq 1} \in \varprojlim  L_{k}$
the limit $\lim_{k \to \infty} B_{k}(g + {\mathfrak m}^{k})$ is well-defined and belongs to ${\mathfrak m}$ for any
$g \in {\mathfrak m}$. We can identify $(B_{k})_{k \geq 1}$ with the expression
\[ \lim_{k \to \infty} B_{k}(x_{1}+ {\mathfrak m}^{k}) \frac{\partial}{\partial x_{1}} + \ldots +
 \lim_{k \to \infty} B_{k}(x_{n}+ {\mathfrak m}^{k}) \frac{\partial}{\partial x_{n}}. \]
 In this way it makes sense to consider the $k$-jet of a formal vector field.
 \begin{defi}
We say that a formal vector field $X \in \hat{\mathfrak X} \cn{n}$ is nilpotent if the first jet
$j^{1} X$ (or $D_{0} X$) is a nilpotent linear transformation.
We denote  by $\hat{\mathfrak X}_{N} \cn{n}$ the subset of nilpotent formal vector fields.
\end{defi}
\begin{rem}
It is a simple exercise to show that an element $\phi = (A_{k})_{k \geq 1} \in \varprojlim D_{k} = \diffh{}{n}$
is unipotent if and only if $A_{k}$ is unipotent for any $k \in {\mathbb N}$. 

Analogously $X = (B_{k})_{k \geq 1} \in \varprojlim  L_{k} = \hat{\mathfrak X} \cn{n}$ is nilpotent 
if and only if $B_{k}$ is nilpotent for any $k \in {\mathbb N}$. 
\end{rem}
\begin{rem}
The exponential establishes a bijection between nilpotent and unipotent matrices.
Since $L_{k}$ is the Lie algebra of $D_{k}$,
the image of the nilpotent elements of $L_{k}$ by the exponential map
coincides with the set of unipotent elements of $D_{k}$. We remind the reader that
the elements of $L_{k}$ and $D_{k}$ are interpreted as linear maps of ${\mathfrak m}/{\mathfrak m}^{k+1}$.
In particular $\mathrm{exp}: \hat{\mathfrak X}_{N} \cn{n} \to \diffh{u}{n}$ is a bijection.
\end{rem}
\begin{defi}
Given $\phi \in \diffh{u}{n}$ we denote by $\log \phi$ the unique element of $\hat{\mathfrak X}_{N} \cn{n}$
such that $\phi = \mathrm{exp} (\log \phi)$.
\end{defi}
\begin{defi}
\label{def:kn}
We define $\hat{K}_{n}$ as the field of fractions of the local ring
${\mathbb C}[[x_{1},\hdots,x_{n}]]$ of formal power series with complex
coefficients.
\end{defi}
\begin{defi}
Let ${\mathfrak g}$ be a complex Lie subalgebra of $\hat{\mathfrak X} \cn{n} \bigotimes_{\mathbb C} \hat{K}_{n}$.
We define $\dim {\mathfrak g}$ as the dimension of the $\hat{K}_{n}$-vector space
${\mathfrak g} \bigotimes_{\mathbb C} \hat{K}_{n}$.
We define ${\mathcal M} ({\mathfrak g}) = \{ g \in \hat{K}_{n}: X(g) \equiv 0 \ \forall X \in {\mathfrak g} \}$, it
is the field of formal meromorphic first integrals of ${\mathfrak g}$.
\end{defi}
\subsection{Pro-algebraic groups}
\label{sec:proalg}
Let us introduce the analogue of algebraic matrix groups for subgroups of $\diffh{}{n}$.
\begin{defi}
Let $G$ be a subgroup of $\diffh{}{n}$. We define $G_{k} = \overline{\{ \phi_{k} : \phi \in G \}}^{z}$
where the Zariski-closure is considered in $\mathrm{GL}({\mathfrak m}/{\mathfrak m}^{k+1})$.
We define the pro-algebraic closure (also Zariski-closure) $\overline{G}^{z}$ of $G$ as the
inverse limit $\varprojlim_{k \in {\mathbb N}} G_{k}$. 
\end{defi}
\begin{rem}
The group $D_{k}$ is algebraic for $k \in {\mathbb N}$. As a consequence
$G_{k}$ is a subgroup of $D_{k}$ for any $k \in {\mathbb N}$ and then
$\overline{G}^{z}$ is a subgroup of $\diffh{}{n} = \varprojlim D_{k}$. In particular the equality
\[ \overline{G}^{z} = \{ \phi \in \diffh{}{n} : \phi_{k} \in G_{k} \ \forall k \in {\mathbb N} \}  \]
holds.
\end{rem}
\begin{rem}
The group $\overline{G}^{z}$ is
closed in the ${\mathfrak m}$-adic topology (the Krull topology)
by construction.
\end{rem}
\begin{defi}
We say that a subgroup $G$ of $\diffh{}{n}$ is pro-algebraic if $G = \overline{G}^{z}$.
\end{defi}
Notice that in \cite{JR:arxivdl} we use the notation $\overline{G}^{(0)}$ instead of $\overline{G}^{z}$.
Since in this paper we use three notions of closure we stress in the new notation that we are referring
to the Zariski-closure.

Let us consider subgroups of formal diffeomorphisms contained in $\diffh{u}{n}$.
The pro-algebraic theory is valid also for general subgroups of $\diffh{}{n}$
but we just apply it in this paper to the unipotent case. Let us introduce some
basic properties of unipotent groups.
\begin{lem}
\label{lem:bas}
Let $G$ be a unipotent subgroup of $\diffh{}{n}$. We have
\begin{itemize}
\item $\overline{G}^{z}$ is contained in $\diffh{u}{n}$.
\item The derived lengths of $G$ and $\overline{G}^{z}$ coincide. In particular $G$ is solvable if and only
if $\overline{G}^{z}$ is solvable.
\item The subset $L(\overline{G}^{z}):= \{ \log \phi : \phi \in \overline{G}^{z} \}$ is a Lie algebra of formal
nilpotent vector fields. Moreover the map $\mathrm{exp} : L(\overline{G}^{z}) \to \overline{G}^{z}$ is a bijection.
\item The derived lengths of $L(\overline{G}^{z})$ and $\overline{G}^{z}$ coincide. In particular $G$ is solvable if and only
if $L(\overline{G}^{z})$ is solvable.
\end{itemize}
\end{lem}
\begin{proof}
All these results are proved in \cite{JR:arxivdl}. They correspond to Lemmas 4, 1 and Propositions 2, 3
respectively.
\end{proof}
\begin{defi}
We say that $L(\overline{G}^{z})$ is the Lie algebra of the pro-algebraic group $\overline{G}^{z}$.
\end{defi}
\begin{rem}
\label{rem:krull}
Since $\overline{G}^{z}$ is
closed in the  Krull topology, hence
$L(\overline{G}^{z})$ is also closed in the Krull topology.
\end{rem}
The algebraic properties of unipotent subgroups of $\diff{}{n}$ are codified in the Lie algebra of
its pro-algebraic closure. The next result is Theorem 6 of \cite{JR:arxivdl},
it displays part of the structure of a Lie algebra of formal vector fields.
\begin{pro}
\label{pro:structure}
Let ${\mathfrak g}$  be a non-trivial complex Lie subalgebra of $\hat{\mathfrak X} \cn{n} \bigotimes_{\mathbb C} \hat{K}_{n}$.
There exist ideals ${\mathfrak h}$, ${\mathfrak j}$ of ${\mathfrak g}$ such that
\begin{itemize}
\item ${\mathfrak j} \subset {\mathfrak h}$ and ${\mathfrak g}/{\mathfrak h}$, ${\mathfrak h}/{\mathfrak j}$ are
abelian Lie algebras.
\item $\dim {\mathfrak g} = \dim {\mathfrak h}$ and $m:= \dim {\mathfrak h} - \dim {\mathfrak j} >0$.
\item ${\mathfrak g}/{\mathfrak h}$ is isomorphic to a complex Lie algebra of $m \times m$ matrices with
coefficients in ${\mathcal M}({\mathfrak g})$.
\item Let $\{ X_{1}, \ldots, X_{m} \}$ be a base of
$({\mathfrak g} \bigotimes_{\mathbb C} \hat{K}_{n})/ ({\mathfrak j} \bigotimes_{\mathbb C} \hat{K}_{n})$
contained in ${\mathfrak h}$. Then given $Z \in {\mathfrak h}$ there exists unique
$h_{1}, \ldots, h_{m} \in{\mathcal M}({\mathfrak g})$ such that
$Z - h_{1} X_{1} - \ldots - h_{m} X_{m} \in {\mathfrak j} \bigotimes_{\mathbb C} \hat{K}_{n}$.
\end{itemize}
\end{pro}
\begin{rem}
\label{rem:precise}
Let us be a little be more precise, the details can be found in the proof of Theorem 6 of \cite{JR:arxivdl}.
The ideal ${\mathfrak j}$ is equal to $({\mathfrak j} \bigotimes_{\mathbb C} \hat{K}_{n}) \cap {\mathfrak g}$.
Moreover we have
\[ \dim {\mathfrak j} = \dim {\mathfrak g}^{(a)}, \ {\mathfrak j} \bigotimes_{\mathbb C} \hat{K}_{n} =
{\mathfrak g}^{(a)} \bigotimes_{\mathbb C} \hat{K}_{n} \ \mathrm{and} \
{\mathfrak j} = ({\mathfrak g}^{(a)} \bigotimes_{\mathbb C} \hat{K}_{n}) \cap {\mathfrak g} \]
for the first $a \in \{1,2\}$
such that $\dim {\mathfrak g}^{(a)} < \dim  {\mathfrak g}$. In particular
${\mathfrak g}/{\mathfrak j}$ is abelian if and only if $a=1$.

Given $Z \in {\mathfrak g}$ we define
\begin{equation}
\label{equ:matrix}
M (Z) =
\left(
\begin{array}{cccc}
X_{1}(h_{1}) & X_{1}(h_{2}) & \hdots & X_{1}(h_{m}) \\
X_{2}(h_{1}) & X_{2}(h_{2}) & \hdots & X_{2}(h_{m}) \\
\vdots & \vdots & & \vdots \\
X_{m}(h_{1}) & X_{m}(h_{2}) & \hdots & X_{m}(h_{m})
\end{array}
\right).
\end{equation}
where $Z - h_{1} X_{1} - \ldots - h_{m} X_{m} \in {\mathfrak j} \bigotimes_{\mathbb C} \hat{K}_{n}$.
The map $M: Z \mapsto M (Z)$ is a morphism of complex Lie algebras from ${\mathfrak g}$
to the Lie algebra of $m \times m$ matrices with coefficients in ${\mathcal M}({\mathfrak g})$.
The kernel of $M$ is equal to ${\mathfrak h}$.
Thus $M: {\mathfrak g}/{\mathfrak h} \to M({\mathfrak g})$ is an isomorphism of complex
Lie algebras.
\end{rem}
\subsection{Structure of a solvable Lie algebra}
In this section we improve Proposition  \ref{pro:structure} and provide a finer classification
of Lie algebras of formal vector fields.

Let $\tilde{\mathcal L}_{1}$ be a solvable
Lie subalgebra of $\hat{\mathfrak X} \cn{n} \bigotimes_{\mathbb C} \hat{K}_{n}$.
We denote
${\mathcal M}_{1}= {\mathcal M} (\tilde{\mathcal L}_{1})$
and ${\mathcal L}_{1} =\tilde{\mathcal L}_{1} \bigotimes_{\mathbb C} {\mathcal M}_{1}$.
Then ${\mathcal L}_{1}$ is a Lie subalgebra of $\hat{\mathfrak X} \cn{n} \bigotimes_{\mathbb C} \hat{K}_{n}$
such that $\ell ({\mathcal L}_{1})= \ell (\tilde{\mathcal L}_{1})$.
Notice that ${\mathcal M}_{1} = {\mathcal M}({\mathcal L}_{1})$.
Our goal is constructing a Lie subalgebra ${\mathcal L}$ of $\hat{\mathfrak X} \cn{n} \bigotimes_{\mathbb C} \hat{K}_{n}$
that contains ${\mathcal L}_{1}$, it has similar algebraic properties as ${\mathcal L}_{1}$
and also
a simpler structure.
In particular we provide a unique decomposition for the elements of ${\mathcal L}$.

We denote ${\mathfrak g} = {\mathcal L}_{1}$. Suppose ${\mathcal L}_{1} \neq 0$.
Consider the notations provided by
Proposition \ref{pro:structure}.
We denote $b_{1} = m = \dim {\mathfrak g} - \dim {\mathfrak j}$.
Since $M$ is ${\mathcal M}_{1}$-linear,
the Lie algebra $M ({\mathcal L}_{1})$ is a ${\mathcal M}_{1}$-vector space of
dimension $e_{1}$ less or equal than $b_{1}^{2}$.
There exist elements $W_{1}^{1}, \ldots, W_{e_{1}}^{1}$ in ${\mathcal L}_{1}$ such that
$\{ M_{W_{1}^{1}}, \ldots, M_{W_{e_{1}}^{1}} \}$ is a basis of $M({\mathcal L}_{1})$.
As a consequence given $X \in {\mathcal L}_{1}$
there exist unique $f_{1}^{1}, \ldots, f_{e_{1}}^{1} \in {\mathcal M}_{1}$
such that
$M (X - \sum_{j=1}^{e_{1}} f_{j}^{1} W_{j}^{1})=0$.
By construction $X - \sum_{j=1}^{e_{1}} f_{j}^{1} W_{j}^{1}$ belongs to ${\mathcal L}_{1}$
and since its image by $M$ vanishes, it belongs to ${\mathfrak h}$.
Hence there exist unique $h_{1}^{1}$, $\ldots$, $h_{b_{1}}^{1} \in {\mathcal M}_{1}$ such that
\[ X  -  \sum_{j=1}^{e_{1}} f_{j}^{1} W_{j}^{1} - \sum_{k=1}^{b_{1}} h_{k}^{1} X_{k}^{1}  \]
belongs to $({\mathfrak j} \bigotimes_{\mathbb C} \hat{K}_{n}) \cap {\mathfrak g}$
or equivalently to ${\mathfrak j}$.
We define $\tilde{\mathcal L}_{2} = {\mathfrak j}$, ${\mathcal M}_{2} = {\mathcal M}(\tilde{\mathcal L}_{2})$
and ${\mathcal L}_{2} = \tilde{\mathcal L}_{2} \bigotimes_{\mathbb C} {\mathcal M}_{2}$.
\begin{defi}
\label{def:mbasis}
We say that $\{Z_{1}^{1}, \hdots, Z_{c_{1}}^{1}\}$ is a ${\mathcal M}'$-{\it basis}
of ${\mathcal L}_{1}$ if it is a subset of ${\mathcal L}_{1}$ such that the
classes of its elements determine a basis of the ${\mathcal M}_{1}$-vector space
${\mathcal L}_{1}/\tilde{\mathcal L}_{2}$.
\end{defi}
\begin{rem}
\label{rem:mbasis}
We define $c_{1} =b_{1} + e_{1}$
and $Z_{j}^{1} = W_{j}^{1}$, $Z_{k+e_{1}}^{1} = X_{k}^{1}$ for
$1 \leq j \leq e_{1}$ and $1 \leq k \leq b_{1}$.
The set $\{Z_{1}^{1}, \ldots, Z_{c_{1}}^{1} \}$ is a
${\mathcal M}'$-{\it basis} of ${\mathcal L}_{1}$.
In particular we proved
\begin{equation}
\label{equ:dimm}
\dim_{{\mathcal M}_{1}} {\mathcal L}_{1}/\tilde{\mathcal L}_{2} = \dim_{{\mathcal M}_{1}} M({\mathcal L}_{1}) +
\dim {\mathcal L}_{1} - \dim {\mathcal L}_{2} \leq n^{2} + n .
\end{equation}
\end{rem}
\begin{rem}
The space ${\mathcal L}_{1}/\tilde{\mathcal L}_{2}$ is a ${\mathcal M}_{1}$-vector space of dimension
$b_{1} +e_{1}$ whereas
$({\mathcal L}_{1} \bigotimes_{\mathbb C} \hat{K}_{n})/(\tilde{\mathcal L}_{2} \bigotimes_{\mathbb C} \hat{K}_{n})$
is a $\hat{K}_{n}$-vector space of dimension $b_{1}$.
The different dimensional type is due to the choice of the base field.
\end{rem}
The Lie algebra $\tilde{\mathcal L}_{2}$ is solvable since  $\tilde{\mathcal L}_{2}  \subset {\mathcal L}_{1}$.
Moreover ${\mathcal L}_{2}$ is a solvable Lie algebra since $\ell (\tilde{\mathcal L}_{2}) = \ell ({\mathcal L}_{2})$.
The property $[{\mathcal L}_{1}, \tilde{\mathcal L}_{2}] \subset \tilde{\mathcal L}_{2}$
implies that the elements in ${\mathcal L}_{1}$ preserve the first integrals of $\tilde{\mathcal L}_{2}$.
More precisely, we have
\begin{equation}
\label{equ:ifi}
0= [Z,X](g) = Z(X(g)) - X(Z(g)) =  Z(X(g))   .
\end{equation}
for all $X \in {\mathcal L}_{1}$, $Z \in \tilde{\mathcal L}_{2}$ and $g \in  {\mathcal M}_{2}$.
We deduce $X({\mathcal M}_{2}) \subset {\mathcal M}_{2}$ for any
$X \in {\mathcal L}_{1}$.
In particular we obtain
$[{\mathcal L}_{1}, {\mathcal L}_{2} ] \subset {\mathcal L}_{2}$.
Since ${\mathcal L}_{1}^{(2)} \subset {\mathcal L}_{2}$,
${\mathcal L}_{1} + {\mathcal L}_{2}$ is a solvable Lie algebra
such that $({\mathcal L}_{1} + {\mathcal L}_{2})^{(2)} \subset {\mathcal L}_{2}$.

 Since ${\mathcal L}_{2}$ is solvable we can repeat the process to obtain
a Lie algebra $\tilde{\mathcal L}_{3}$ if ${\mathcal L}_{2} \neq 0$. Then we define
${\mathcal M}_{3}={\mathcal M}(\tilde{\mathcal L}_{3})$ and
${\mathcal L}_{3}= \tilde{\mathcal L}_{3} \bigotimes_{\mathbb C} {\mathcal M}_{3}$.
We obtain a ${\mathcal M}'$-basis $\{ Z_{1}^{2} , \ldots, Z_{c_{2}}^{2} \}$ for ${\mathcal L}_{2}$ analogously
as for ${\mathcal L}_{1}$.
Given any $X \in {\mathcal L}_{2}$ there exist unique
$\gamma_{1}^{2}, \ldots, \gamma_{c_{2}}^{2} \in {\mathcal M}_{2}$
such that $X - \sum_{j=1}^{c_{2}} \gamma_{j}^{2} Z_{j}^{2} \in \tilde{\mathcal L}_{3}$.
We continue to obtain Lie algebras
\[ {\mathcal L}_{1}, \hdots, {\mathcal L}_{m}, {\mathcal L}_{m+1}=0,  {\mathcal L}_{m+2}=0, \ldots \]
where ${\mathcal L}_{m} \neq 0$.
Notice that such $m$ exists since $\dim {\mathcal L}_{j+1} < \dim {\mathcal L}_{j}$ if
${\mathcal L}_{j} \neq 0$.
\begin{defi}
We say that ${\mathcal L}= \sum_{j=1}^{\infty} {\mathcal L}_{j}= {\mathcal L}_{1} + \ldots + {\mathcal L}_{m}$
is the {\it extension Lie algebra} associated to $\tilde{\mathcal L}_{1}$.
We define ${\mathcal L}_{j}=0$ for $j \geq 1$ and ${\mathcal L}=0$ if $\tilde{\mathcal L}_{1}=0$.
\end{defi}
\begin{rem}
The sum $\sum_{j=1}^{\infty} {\mathcal L}_{j}$ is not a direct sum unless
${\mathcal L}_{j} =0$ for $j \geq 2$. We did not show that  ${\mathcal L}$
is a Lie algebra yet, it will be done in Proposition \ref{pro:sla}.
\end{rem}
\begin{defi}
We denote ${\mathcal M}_{j}= {\mathcal M}({\mathcal L}_{j})$.
\end{defi}
By construction we obtain ${\mathcal L}_{j}^{(2)} \subset {\mathcal L}_{j+1}$ and
$[{\mathcal L}_{j}, {\mathcal L}_{j+1}] \subset {\mathcal L}_{j+1}$ for any $1 \leq j \leq m$.
Next we generalize the latter property.
\begin{lem}
Let $1 \leq j < k \leq m+1$. We have
$[{\mathcal L}_{j}, {\mathcal L}_{k}] \subset {\mathcal L}_{k}$.
\end{lem}
\begin{proof}
The result holds for $k=j+1$ by construction.
Let us prove that if it holds for some $k \geq j+1$ so it does for $k+1$.
Since $[{\mathcal L}_{j}, {\mathcal L}_{k}] \subset {\mathcal L}_{k}$ we obtain
$[{\mathcal L}_{j}, ({\mathcal L}_{k})^{(1)}] \subset ({\mathcal L}_{k})^{(1)}$ and then
$[{\mathcal L}_{j}, ({\mathcal L}_{k})^{(2)}] \subset ({\mathcal L}_{k})^{(2)}$ by the
Jacobi identity. We have
$\tilde{\mathcal L}_{k+1} \bigotimes_{\mathbb C} \hat{K}_{n} = ({\mathcal L}_{k})^{(a)} \bigotimes_{\mathbb C} \hat{K}_{n}$
for some $a \in \{1,2\}$
by Remark \ref{rem:precise}. Since
\[ [{\mathcal L}_{j},({\mathcal L}_{k})^{(a)} \bigotimes_{\mathbb C} \hat{K}_{n}] \subset
({\mathcal L}_{k})^{(a)} \bigotimes_{\mathbb C} \hat{K}_{n} \]
for $a \in \{1,2\}$, we deduce
$[{\mathcal L}_{j},\tilde{\mathcal L}_{k+1} \bigotimes_{\mathbb C} \hat{K}_{n}] \subset
\tilde{\mathcal L}_{k+1} \bigotimes_{\mathbb C} \hat{K}_{n}$.
Notice that
$(\tilde{\mathcal L}_{k+1} \bigotimes_{\mathbb C} \hat{K}_{n}) \cap {\mathcal L}_{k} = \tilde{\mathcal L}_{k+1}$
by Remark \ref{rem:precise}.
Since $[{\mathcal L}_{j}, {\mathcal L}_{k}] \subset {\mathcal L}_{k}$ we obtain
$[{\mathcal L}_{j}, \tilde{\mathcal L}_{k+1}] \subset \tilde{\mathcal L}_{k+1}$.
This property implies $X({\mathcal M}_{k+1}) \subset {\mathcal M}_{k+1}$ for any $X \in {\mathcal L}_{j}$
(cf. Equation (\ref{equ:ifi})).
We deduce $[{\mathcal L}_{j}, {\mathcal L}_{k+1}] \subset {\mathcal L}_{k+1}$.
\end{proof}
Our next goal is showing that the extension ${\mathcal L}$ of $\tilde{\mathcal L}_{1}$ has a simple
structure.
\begin{pro}
\label{pro:sla}
Let ${\mathcal L}$ be the extension Lie algebra associated to a solvable complex Lie subalgebra
$\tilde{\mathcal L}_{1}$ of $\hat{\mathfrak X} \cn{n} \bigotimes_{\mathbb C} \hat{K}_{n}$. Then
${\mathcal L}$ is a solvable Lie algebra.
\end{pro}
\begin{proof}
Since ${\mathcal L}_{j}^{(2)} \subset {\mathcal L}_{j+1}$ and
$[{\mathcal L}_{j}, {\mathcal L}_{k}] \subset {\mathcal L}_{k}$
for all $1 \leq j \leq m$ and $j \leq k$, we deduce
that ${\mathcal L}$ is a Lie algebra such that
\[ ({\mathcal L}_{l} + \ldots + {\mathcal L}_{m})^{(2)} \subset {\mathcal L}_{l+1} + \ldots + {\mathcal L}_{m} \]
for any $1 \leq l \leq m$. In particular $\ell ({\mathcal L})$ is less or equal than $2m$.
\end{proof}
The previous construction provides a sequence $(Z_{k}^{j})_{1 \leq j \leq m, \ 1 \leq k \leq c_{j}}$,
where $\{ Z_{1}^{j}, \ldots, Z_{c_{j}}^{j} \}$ is a ${\mathcal M}'$-basis of ${\mathcal L}_{j}$
for any $1 \leq j \leq m$. Let us see that such sequence is associated to a unique decomposition of the
elements of the extension Lie algebra.
\begin{pro}
\label{pro:uni}
Let ${\mathcal L}$ be the extension Lie algebra associated to a solvable Lie subalgebra
$\tilde{\mathcal L}_{1}$ of $\hat{\mathfrak X} \cn{n} \bigotimes_{\mathbb C} \hat{K}_{n}$.
Given an element $X$ of ${\mathcal L}$, it can be written uniquely in the form
$\sum_{j=1}^{m} \sum_{k=1}^{c_{j}} \gamma_{k}^{j} Z_{k}^{j}$
where $\gamma_{k}^{j} \in {\mathcal M}_{j}$ for all $1 \leq j \leq m$ and $1 \leq k \leq c_{j}$.
Indeed the elements of ${\mathcal L}$ are exactly the elements of
$\hat{\mathfrak X} \cn{n} \bigotimes_{\mathbb C} \hat{K}_{n}$ that can be written in the previous
form.
\end{pro}
\begin{proof}
Any element of the form $\sum_{j=1}^{m} \sum_{k=1}^{c_{j}} \gamma_{k}^{j} Z_{k}^{j}$
(where $\gamma_{k}^{j} \in {\mathcal M}_{j}$ for all $1 \leq j \leq m$ and $1 \leq k \leq c_{j}$)
belongs to ${\mathcal L}$ since
$\sum_{k=1}^{c_{j}} \gamma_{k}^{j} Z_{k}^{j} \in {\mathcal L}_{j}$ for any $1 \leq j \leq m$.

We denote ${\mathcal K}_{l} = {\mathcal L}_{l} + \ldots + {\mathcal L}_{m}$.
Let $1 \leq l \leq m$.
Let us prove that any element $X$ of  ${\mathcal K}_{l} $
can be written uniquely in the form
$\sum_{j=l}^{m} \sum_{k=1}^{c_{j}} \gamma_{k}^{j} Z_{k}^{j}$ where
$\gamma_{k}^{j} \in {\mathcal M}_{j}$ for all $l \leq j \leq m$ and $1 \leq k \leq c_{j}$.
The result is obvious for $l=m$ by construction. Let us show that if the result holds for
$2 \leq l+1 \leq m$ then so it does for $l$.
Consider $X \in {\mathcal K}_{l} $.
It is of the form $\sum_{j=l}^{m} X_{j}$ where
$X_{j} \in {\mathcal L}_{j}$ for any $l \leq j \leq m$.
If the decomposition $X=\sum_{j=l}^{m} \sum_{k=1}^{c_{j}} \gamma_{k}^{j} Z_{k}^{j}$ exists then
$\sum_{k=1}^{c_{l}} \gamma_{k}^{l} Z_{k}^{l}$ is a formal vector field with
$ \gamma_{k}^{l} \in {\mathcal M}_{l}$ for $1 \leq k \leq c_{l}$ such that
$X - \sum_{k=1}^{c_{l}} \gamma_{k}^{l} Z_{k}^{l} \in {\mathcal K}_{l+1} $.
We use the previous condition to define $\sum_{k=1}^{c_{l}} \gamma_{k}^{l} Z_{k}^{l}$.
Indeed since the condition is equivalent to
$X_{l} -  \sum_{k=1}^{c_{l}} \gamma_{k}^{l} Z_{k}^{l} \in \tilde{\mathcal L}_{l+1}$,
we deduce that $\sum_{k=1}^{c_{l}} \gamma_{k}^{l} Z_{k}^{l}$ is unique with the required properties.
The result is a consequence of applying the induction hypothesis to
$X -  \sum_{k=1}^{c_{l}} \gamma_{k}^{l} Z_{k}^{l}$.
\end{proof}
\begin{defi}
We say that ${\mathcal B}:=\{ Z_{1}^{1}, \ldots, Z_{c_{1}}^{1}, \ldots, Z_{1}^{m}, \ldots, Z_{c_{m}}^{m} \}$
is a ${\mathcal M}$-{\it basis} of ${\mathcal L} = {\mathcal L}_{1} + \ldots + {\mathcal L}_{m}$ if
$\{Z_{1}^{j}, \ldots, Z_{c_{j}}^{j}\}$ is a ${\mathcal M}'$-basis of ${\mathcal L}_{j}$ for any
$1 \leq j \leq m$. Let $X \in {\mathcal L}$; the unique expression provided by
Proposition \ref{pro:uni} is called a
${\mathcal M}$-{\it decomposition} of $X$ with respect to ${\mathcal L}$  and ${\mathcal B}$.
\end{defi}
Resuming we replace a solvable Lie subalgebra $\tilde{\mathcal L}_{1}$ of
$\hat{\mathfrak X} \cn{n} \bigotimes_{\mathbb C} \hat{K}_{n}$ with a solvable
Lie algebra ${\mathcal L}$ that has a much simpler structure since
its elements admit the unique expression provided by a ${\mathcal M}$-decomposition.
Let us remark that a vector field or a diffeomorphism that normalizes
$\tilde{\mathcal L}_{1}$ also normalizes ${\mathcal L}$, making the extension
Lie algebra suitable for the study of algebraic and geometrical problems.
\subsection{Increasing sequences of subgroups of a solvable group}
Let $\tilde{\mathcal L}_{1}^{0} \subset \tilde{\mathcal L}_{1}^{1} \subset \ldots \subset \tilde{\mathcal L}_{1}^{j} \subset \ldots$
be an increasing sequence of solvable Lie subalgebras of
$\hat{\mathfrak X} \cn{n} \bigotimes_{\mathbb C} \hat{K}_{n}$.
We can associate Lie algebras $\tilde{\mathcal L}_{k}^{j}$, ${\mathcal L}_{k}^{j}$ and
${\mathcal L}^{j}={\mathcal L}_{1}^{j} + \ldots + {\mathcal L}_{m_{j}}^{j}$ to $\tilde{\mathcal L}_{1}^{j}$
as described in the previous section. We place the superindex $j$ to all the objects
associated to $\tilde{\mathcal L}_{1}^{j}$.
\begin{rem}
\label{rem:extg}
We will apply the results in this section to sequences
$J(0) \subset J(1) \subset \ldots$  of subgroups of a  solvable group $\Gamma$
contained in $\mathrm{Diff}_{u}({\mathbb C}^{n},0)$.
We define $\tilde{\mathcal L}_{1}^{j} = L(\overline{J(j)}^{z})$.
Since $J(j) \subset J(j+1)$ for $j \geq 0$, we obtain $\overline{J(j)}^{z} \subset \overline{J(j+1)}^{z}$
and
\[ \tilde{\mathcal L}_{1}^{j} =
L(\overline{J(j)}^{z}) \subset L(\overline{J(j+1)}^{z}) = \tilde{\mathcal L}_{1}^{j+1}  \]
for $j \geq 0$. In particular we obtain an increasing sequence
$\tilde{\mathcal L}_{1}^{0} \subset \tilde{\mathcal L}_{1}^{1} \subset \ldots$ of solvable Lie algebras
(cf. Lemma \ref{lem:bas}).
\end{rem}
We will associate to ${\mathcal L}^{j}$ a list of invariants. We will see that
the sequence of lists for $j \geq 0$ is increasing in the lexicographical order.
The equality of the invariants associated to ${\mathcal L}^{j}$ and ${\mathcal L}^{j+1}$
implies ${\mathcal L}^{j} = {\mathcal L}^{j+1}$. We will obtain ${\mathcal L}^{j} = {\mathcal L}^{j+1}$
for some $j \geq 0$ by showing that the list takes finitely many values.
\begin{lem}
There exists $0 \leq j \leq n$ such that $\dim {\mathcal L}_{1}^{j} = \dim {\mathcal L}_{1}^{j+1}$.
Moreover there exists $0 \leq k < (n+1)^{3}$ such that
$\dim {\mathcal L}_{1}^{k} = \dim {\mathcal L}_{1}^{k+1}$,
$\dim ({\mathcal L}_{1}^{k} )^{(1)}= \dim ({\mathcal L}_{1}^{k+1})^{(1)}$
and $\dim ({\mathcal L}_{1}^{k})^{(2)} = \dim ({\mathcal L}_{1}^{k+1})^{(2)}$.
\end{lem}
\begin{proof}
 Since $\dim {\mathcal L}_{1}^{j}  = \dim \tilde{\mathcal L}_{1}^{j}$
for $j \geq 1$, we obtain an increasing sequence
\[ 0 \leq \dim {\mathcal L}_{1}^{1} \leq  \dim{\mathcal L}_{1}^{2} \leq \ldots \leq
\dim {\mathcal L}_{1}^{j} \leq \dim {\mathcal L}_{1}^{j+1} \leq \ldots \leq n. \]
Clearly there exists $0 \leq j \leq n$ such that $\dim {\mathcal L}_{1}^{j} = \dim {\mathcal L}_{1}^{j+1}$.

Notice that $(\tilde{\mathcal L}_{1}^{j})^{(k)} \subset (\tilde{\mathcal L}_{1}^{j+1})^{(k)}$
and $\dim ({\mathcal L}_{1}^{j})^{(k)} = \dim (\tilde{\mathcal L}_{1}^{j})^{(k)}$
for all $j \geq 0$ and $k \geq 0$. As a consequence
%
$(\dim {\mathcal L}_{1}^{j},\dim ({\mathcal L}_{1}^{j} )^{(1)},\dim ({\mathcal L}_{1}^{j} )^{(2)})$
is increasing in every component and in particular for the lexicographical order.
Since there are at most $(n+1)^{3}$ $3$-uples, they coincide for some $0 \leq k < (n+1)^{3}$.
\end{proof}
\begin{defi}
We define $ I(j,k) =(0,0,0,0)$ if ${\mathcal L}_{k}^{j} = 0$ and
\[  I(j,k) = (dim  ({\mathcal L}_{k}^{j}) , dim  ({\mathcal L}_{k}^{j})^{(1)}, dim  ({\mathcal L}_{k}^{j})^{(2)} ,
dim_{{\mathcal M}_{k}^{j}} {\mathcal L}_{k}^{j}/\tilde{\mathcal L}_{k+1}^{j})  \]
otherwise.
We define the sequence
$ I(j) = (  I(j,1) , \ldots,  I(j,n) )$.
\end{defi}
Let us study the behavior of the sequence $( I(j,1))_{j \geq 0}$ and its impact on the relation
between ${\mathcal L}_{2}^{j}$ and ${\mathcal L}_{2}^{j+1}$.
\begin{lem}
\label{lem:aux1}
Let $j \geq 0$ such that $\dim ({\mathcal L}_{1}^{j})^{(k)} = \dim ({\mathcal L}_{1}^{j+1})^{(k)}$
for any $k \in \{0,1,2\}$. Then
\begin{itemize}
\item ${\mathcal L}_{1}^{j} \subset {\mathcal L}_{1}^{j+1}$, $\tilde{\mathcal L}_{2}^{j} \subset \tilde{\mathcal L}_{2}^{j+1}$,
${\mathcal L}_{2}^{j} \subset {\mathcal L}_{2}^{j+1}$
and $\dim {\mathcal L}_{2}^{j}= \dim {\mathcal L}_{2}^{j+1}$.
\item ${\mathcal M}_{1}^{j} = {\mathcal M}_{1}^{j+1}$ and ${\mathcal M}_{2}^{j} = {\mathcal M}_{2}^{j+1}$.
\item $ I(j,1)  \leq  I(j+1,1) $.
\end{itemize}
\end{lem}
\begin{proof}
We deduce $\tilde{\mathcal L}_{1}^{j} \bigotimes_{\mathbb C} \hat{K}_{n}=
\tilde{\mathcal L}_{1}^{j+1} \bigotimes_{\mathbb C} \hat{K}_{n}$ since
$\tilde{\mathcal L}_{1}^{j} \subset \tilde{\mathcal L}_{1}^{j+1}$ and
$\dim \tilde{\mathcal L}_{1}^{j} = \dim \tilde{\mathcal L}_{1}^{j+1}$.
In particular we obtain ${\mathcal M}_{1}^{j}={\mathcal M}_{1}^{j+1}$.
Thus ${\mathcal L}_{1}^{j}$ is contained in ${\mathcal L}_{1}^{j+1}$.
Moreover
$({\mathcal L}_{1}^{j})^{(k)} \subset ({\mathcal L}_{1}^{j+1})^{(k)}$
for any $k \in \{0,1,2\}$.

We have $\dim \tilde{\mathcal L}_{2}^{l} = \dim ({\mathcal L}_{1}^{l})^{(k_{l})}$ and
$\tilde{\mathcal L}_{2}^{l} = (({\mathcal L}_{1}^{l})^{(k_{l})} \bigotimes_{\mathbb C} \hat{K}_{n}) \cap {\mathcal L}_{1}^{l}$
for the first $k_{l} \in \{1,2\}$ such that $\dim ({\mathcal L}_{1}^{l})^{(k_{l})} < \dim {\mathcal L}_{1}^{l}$ for $l \geq 0$
by Remark \ref{rem:precise}. Let $r:=k_{j}=k_{j+1}$; we obtain
$\tilde{\mathcal L}_{2}^{j}  \subset \tilde{\mathcal L}_{2}^{j+1}$ and
\[ \dim {\mathcal L}_{2}^{j}= \dim \tilde{\mathcal L}_{2}^{j} = \dim ({\mathcal L}_{1}^{j})^{(r)} =
\dim ({\mathcal L}_{1}^{j+1})^{(r)} =   \dim \tilde{\mathcal L}_{2}^{j+1} = \dim {\mathcal L}_{2}^{j+1} . \]
The properties $\tilde{\mathcal L}_{2}^{j}  \subset \tilde{\mathcal L}_{2}^{j+1}$
and $\dim {\mathcal L}_{2}^{j}=\dim {\mathcal L}_{2}^{j+1}$ imply
${\mathcal M}_{2}^{j}={\mathcal M}_{2}^{j+1}$.
Hence we get
${\mathcal L}_{2}^{j} \subset {\mathcal L}_{2}^{j+1}$.

The natural map $\varpi_{j}: {\mathcal L}_{1}^{j}/\tilde{\mathcal L}_{2}^{j} \to {\mathcal L}_{1}^{j+1}/\tilde{\mathcal L}_{2}^{j+1}$
is well-defined and ${\mathcal M}_{1}$-linear where ${\mathcal M}_{1}:={\mathcal M}_{1}^{j}={\mathcal M}_{1}^{j+1}$.
Since
\[ \tilde{\mathcal L}_{2}^{j} \bigotimes_{\mathbb C} \hat{K}_{n}=\tilde{\mathcal L}_{2}^{j+1} \bigotimes_{\mathbb C} \hat{K}_{n},
\ \tilde{\mathcal L}_{2}^{j} =  (\tilde{\mathcal L}_{2}^{j} \bigotimes_{\mathbb C} \hat{K}_{n}) \cap {\mathcal L}_{1}^{j} \]
and $\tilde{\mathcal L}_{2}^{j+1} =  (\tilde{\mathcal L}_{2}^{j+1} \bigotimes_{\mathbb C} \hat{K}_{n}) \cap {\mathcal L}_{1}^{j+1}$,
the map $\varpi_{j}$ is injective and we obtain
$\dim_{{\mathcal M}_{1}} {\mathcal L}_{1}^{j}/\tilde{\mathcal L}_{2}^{j} \leq
\dim_{{\mathcal M}_{1}} {\mathcal L}_{1}^{j+1}/\tilde{\mathcal L}_{2}^{j+1}$.
\end{proof}
\begin{cor}
\label{cor:ind}
The sequence $( I(j,1) )_{j \geq 0}$ is increasing in the lexicographical order.
Moreover $ I(j,1)  =  I(j+1,1) $ implies ${\mathcal L}_{2}^{j} \subset {\mathcal L}_{2}^{j+1}$.
\end{cor}
We define $C=(n+1)^{3}(n^{2}+1)$. We obtain
\begin{lem}
\label{lem:aux2}
Suppose $ I(l,1) =  I(l+1,1)$. Then we have
\begin{itemize}
\item ${\mathcal L}_{1}^{l} \subset {\mathcal L}_{1}^{l+1}$ and
$dim  ({\mathcal L}_{1}^{l})^{(k)} = \dim ({\mathcal L}_{1}^{l+1})^{(k)}$ for any $k \in \{0,1,2\}$.
\item $\dim {\mathcal L}_{2}^{l}= \dim {\mathcal L}_{2}^{l+1}$, $\tilde{\mathcal L}_{2}^{l} \subset \tilde{\mathcal L}_{2}^{l+1}$
and ${\mathcal L}_{2}^{l} \subset {\mathcal L}_{2}^{l+1}$.
\item ${\mathcal M}_{1}^{l} = {\mathcal M}_{1}^{l+1}$ and ${\mathcal M}_{2}^{l} = {\mathcal M}_{2}^{l+1}$.
\item Any   ${\mathcal M}'$-basis of ${\mathcal L}_{1}^{l}$ is
a  ${\mathcal M}'$-basis of ${\mathcal L}_{1}^{l+1}$.
\end{itemize}
Moreover, there exists $0 \leq l < C$ such that $ I(l,1) =  I(l+1,1)$.
\end{lem}
\begin{proof}
All the items except the last one are consequence of Lemma \ref{lem:aux1}.
Since $\varpi_{l}: {\mathcal L}_{1}^{l}/\tilde{\mathcal L}_{2}^{l} \to {\mathcal L}_{1}^{l+1}/\tilde{\mathcal L}_{2}^{l+1}$
is injective by the proof of Lemma \ref{lem:aux1},
$\dim_{{\mathcal M}_{1}^{l}}  {\mathcal L}_{1}^{l}/\tilde{\mathcal L}_{2}^{l} =
\dim_{{\mathcal M}_{1}^{l+1}}  {\mathcal L}_{1}^{l+1}/\tilde{\mathcal L}_{2}^{l+1}$
and ${\mathcal M}_{1}^{l}={\mathcal M}_{1}^{l+1}$,
the map $\varpi_{l}$ is an isomorphism. Hence a
${\mathcal M}'$-basis of ${\mathcal L}_{1}^{l}$ is a ${\mathcal M}'$-basis
of ${\mathcal L}_{1}^{l+1}$.

Suppose that the first three coordinates of $I(q,1)$ does not change for
$b \leq q \leq c$.
The sequence
$(\dim_{{\mathcal M}_{1}} {\mathcal L}_{1}^{j}/\tilde{\mathcal L}_{2}^{j})_{b \leq j \leq c}$
is increasing by Lemma \ref{lem:aux1} where ${\mathcal M}_{1}={\mathcal M}_{1}^{j}$ for $b \leq j \leq c$.
Since
\[ 0 \leq \dim_{{\mathcal M}_{1}} {\mathcal L}_{1}^{k}/\tilde{\mathcal L}_{2}^{k}  -
\dim_{{\mathcal M}_{1}} {\mathcal L}_{1}^{j}/\tilde{\mathcal L}_{2}^{j} \leq n^{2} \]
for all $b \leq j \leq k \leq c$ by Equation (\ref{equ:dimm}),
the sequence $(\dim_{{\mathcal M}_{1}} {\mathcal L}_{1}^{j}/\tilde{\mathcal L}_{2}^{j})_{b \leq j \leq c}$
takes at most $n^{2}+1$ values. Hence the sequence $( I(j,1) )_{j \geq 0}$ takes at most
$C$ values. In particular
there exists $0 \leq l < C$ such that $ I(l,1)= I(l+1,1)$.
\end{proof}
By applying Lemma \ref{lem:aux2} at most $n$ times we obtain
\begin{pro}
\label{pro:staext}
Let $\tilde{\mathcal L}_{1}^{0} \subset \tilde{\mathcal L}_{1}^{1} \subset \ldots$ be an
increasing sequence of solvable Lie subalgebras of
$\hat{\mathfrak X} \cn{n} \bigotimes_{\mathbb C} \hat{K}_{n}$.
The sequence $( I(j))_{j \geq 0}$ is increasing in the lexicographical order. Moreover
there exists $0 \leq q < C^{n}$ such that $ I(q)= I(q+1)$. In particular we obtain
\begin{itemize}
\item ${\mathcal L}^{q} = {\mathcal L}_{1}^{q} + \ldots + {\mathcal L}_{m}^{q}$ and
${\mathcal L}^{q+1} = {\mathcal L}_{1}^{q} + \ldots + {\mathcal L}_{m}^{q+1}$.
\item  $\dim ({\mathcal L}_{j}^{q})^{(k)} = \dim ({\mathcal L}_{j}^{q+1})^{(k)}$
for all $1 \leq j \leq m$ and $k \in \{0,1,2\}$.
\item $\tilde{\mathcal L}_{j}^{q} \subset \tilde{\mathcal L}_{j}^{q+1}$,
${\mathcal L}_{j}^{q} \subset {\mathcal L}_{j}^{q+1}$ and
${\mathcal M}_{j}^{q}= {\mathcal M}_{j}^{q+1}$ for any $1 \leq j \leq m$.
\item Any ${\mathcal M}'$-basis of ${\mathcal L}_{j}^{q}$ is a ${\mathcal M}'$-basis of
${\mathcal L}_{j}^{q+1}$ for any $1 \leq j \leq m$.
\item Any ${\mathcal M}$-basis of ${\mathcal L}^{q}$ is a ${\mathcal M}$-basis of
${\mathcal L}^{q+1}$.
\item ${\mathcal L}^{q}={\mathcal L}^{q+1}$.
\end{itemize}
\end{pro}
\begin{proof}
Corollary \ref{cor:ind} implies that we can apply iteratively Lemma \ref{lem:aux2}
to obtain the increasing nature of $( I(j))_{j \geq 0}$ and
the existence of $q$ satisfying all
properties but the last one.   Let us prove ${\mathcal L}^{q}={\mathcal L}^{q+1}$.
We choose a ${\mathcal M}$-basis ${\mathcal B}$ of ${\mathcal L}^{q}$.
It is also a ${\mathcal M}$-basis of ${\mathcal L}^{q+1}$.
Since ${\mathcal M}_{j}^{q}={\mathcal M}_{j}^{q+1}$ for any $1 \leq j \leq m$,
the set of ${\mathcal M}$-decompositions with respect to ${\mathcal L}^{q}$ and
${\mathcal B}$ coincides with the set of
${\mathcal M}$-decompositions with respect to ${\mathcal L}^{q+1}$ and
${\mathcal B}$. Proposition \ref{pro:uni} implies ${\mathcal L}^{q}={\mathcal L}^{q+1}$.
\end{proof}
\begin{rem}
Notice that if $ I(j,k)= I(j+1,k)$ then the first coordinates of
$ I(j,k+1)$ and $ I(j+1,k+1)$ coincide by Lemma \ref{lem:aux2}.
As a consequence we can replace $C^{n}$ in Proposition \ref{pro:staext}
with $n (n^{2} (n^{2}+1))^{n}$ and the result still holds.
\end{rem}
\subsection{Solvability of $p$-pseudo-solvable subgroups}
This section is devoted to prove the following theorem:
\begin{teo}
\label{teo:psis}
Let $G \subset \mathrm{Diff}_{u}({\mathbb C}^{n},0)$ be a finitely generated group  
and $p \geq C^{n}$.
Suppose that $G$ is $p$-pseudo-solvable for some finite generator set  ${\mathcal S}$.
Then $G$ is solvable.
\end{teo}
Let $m_{0} \geq 0$ be the first index such that ${\mathcal S}(m_{0})=\{Id\}$.
We define
\[ {\mathcal S}(j,k) = {\mathcal S}(j) \cup {\mathcal S}(j+1) \cup \ldots \cup {\mathcal S}(k), \]
$G(j,k) = \langle {\mathcal S}(j,k) \rangle$ and
$\Gamma (l) = G(l, m_{0})$
for $0 \leq j \leq k$ and $0 \leq l \leq m_{0}$. Our goal is proving that
$\Gamma (0)$ is solvable. It is obvious that $\Gamma (m_{0})$ is solvable
since it is the trivial group. We will show that whenever $\Gamma (m+1)$ is
solvable for $0 \leq m \leq m_{0}-1$ then $\Gamma (m)$ is solvable.

Let  ${\mathcal L}^{j}$ be the extension Lie algebra associated to
$G(m+1, j)$ (where $\tilde{\mathcal L}_{1}^{j} := L(\overline{G(m+1, j)}^{z})$) for $j \geq m+1$.
Since $(G(m+1,j))_{j \geq m+1}$ is an increasing
sequence of subgroups of the solvable group $\overline{\Gamma(m+1)}^{z}$ (cf. Lemma \ref{lem:bas}), 
there exists $m+1 \leq q \leq m+p$ such that
${\mathcal L}^{q} = {\mathcal L}^{q+1}$ by Remark \ref{rem:extg} and Proposition \ref{pro:staext}.
Moreover we have ${\mathcal L}^{q} = {\mathcal L}_{1}^{q} + \ldots + {\mathcal L}_{m}^{q}$,
${\mathcal L}^{q+1} = {\mathcal L}_{1}^{q+1} + \ldots + {\mathcal L}_{m}^{q+1}$
and ${\mathcal L}^{q}$ and ${\mathcal L}^{q+1}$ satisfy all conditions in
Proposition \ref{pro:staext}. We denote ${\mathcal L}={\mathcal L}^{q}$ and
${\mathcal M}_{j} = {\mathcal M}_{j}^{q}$.

The idea is extending $\varphi G(m+1,q) \varphi^{-1} \subset G(m+1,q+1)$
for $\varphi \in {\mathcal S}(q-p,q)$ to the extension Lie algebras ${\mathcal L}^{q}$ and ${\mathcal L}^{q+1}$.
This is natural since
we enlarge the set of coefficients by adding first integrals of Lie subalgebras canonically associated
to the initial Lie algebra.
\begin{lem}
\label{lem:nor1}
We have ${\mathcal M}_{j} \circ \varphi = {\mathcal M}_{j}$,
$\varphi_{*}  \tilde{\mathcal L}_{j}^{q} \subset \tilde{\mathcal L}_{j}^{q+1}$ and
$\varphi_{*}  {\mathcal L}_{j}^{q} \subset {\mathcal L}_{j}^{q+1}$
for all $1 \leq j \leq m$ and $\varphi \in {\mathcal S}(q-p,q)$.
In particular $\varphi_{*} {\mathcal L}={\mathcal L}$ for any
$\varphi \in {\mathcal S}(q-p,q)$.
\end{lem}
\begin{proof}
We have
$\varphi G(m+1,q) \varphi^{-1} \subset G(m+1,q+1)$ for any $\varphi \in {\mathcal S}(q-p,q)$ by
the definition of the sets ${\mathcal S}(l)$ for $l \geq 0$.
We obtain
\[ \varphi \overline{G(m+1,q)}^{z}  \varphi^{-1} \subset \overline{G(m+1,q+1)}^{z} \]
and
$\varphi_{*} L(\overline{G(m+1,q)}^{z}) \subset L(\overline{G(m+1,q+1)}^{z})$
for any $\varphi \in {\mathcal S}(q-p,q)$.
Equivalently we have $\varphi_{*} \tilde{\mathcal L}_{1}^{q} \subset  \tilde{\mathcal L}_{1}^{q+1}$
for any $\varphi \in {\mathcal S}(q-p,q)$.

Since $\dim \tilde{\mathcal L}_{1}^{q} = \dim \tilde{\mathcal L}_{1}^{q+1}$, we deduce
$\varphi_{*} (\tilde{\mathcal L}_{1}^{q} \bigotimes_{\mathbb C} \hat{K}_{n}) =
\tilde{\mathcal L}_{1}^{q+1} \bigotimes_{\mathbb C} \hat{K}_{n}$. In particular we obtain
${\mathcal M}_{1} \circ \varphi = {\mathcal M}_{1}$ for any $\varphi \in {\mathcal S}(q-p,q)$.
The properties $\varphi_{*} \tilde{\mathcal L}_{1}^{q} \subset  \tilde{\mathcal L}_{1}^{q+1}$ and
$\varphi_{*} {\mathcal M}_{1}= {\mathcal M}_{1}$ imply
$\varphi_{*} {\mathcal L}_{1}^{q} \subset  {\mathcal L}_{1}^{q+1}$ for any $\varphi \in {\mathcal S}(q-p,q)$.

Fix $\varphi \in {\mathcal S}(q-p,q)$.
Let us prove next that
$\varphi_{*} {\mathcal L}_{j}^{q} \subset {\mathcal L}_{j}^{q+1}$ implies
$\varphi_{*} {\mathcal M}_{j+1}= {\mathcal M}_{j+1}$,
$\varphi^{*} \tilde{\mathcal L}_{j+1}^{q} \subset \tilde{\mathcal L}_{j+1}^{q+1}$ and
$\varphi_{*} {\mathcal L}_{j+1}^{q} \subset {\mathcal L}_{j+1}^{q+1}$ for any $1 \leq j < m$.
Since $\varphi_{*} {\mathcal L}_{j}^{q} \subset {\mathcal L}_{j}^{q+1}$ we obtain
$\varphi_{*} ({\mathcal L}_{j}^{q})^{(k)} \subset ({\mathcal L}_{j}^{q+1})^{(k)}$ for $k \geq 0$.
Since $\dim ({\mathcal L}_{j}^{q})^{(k)} = \dim ({\mathcal L}_{j}^{q+1})^{(k)}$ for $k \in \{0,1,2\}$,
the equality $({\mathcal L}_{j}^{q})^{(k)} \bigotimes_{\mathbb C} \hat{K}_{n}=
({\mathcal L}_{j}^{q+1})^{(k)} \bigotimes_{\mathbb C} \hat{K}_{n}$ holds for $k \in \{0,1,2\}$.
There exists $k' \in \{1,2\}$ such that
\[ \tilde{\mathcal L}_{j+1}^{l} = {\mathcal L}_{j}^{l}  \cap (({\mathcal L}_{j}^{l})^{(k')} \bigotimes_{\mathbb C} \hat{K}_{n})
\ \mathrm{and} \ {\mathcal M}_{j+1}^{l}= {\mathcal M}(({\mathcal L}_{j}^{l})^{(k')} )  \]
for any $l \in \{q,q+1\}$. Since $\varphi_{*} (({\mathcal L}_{j}^{q})^{(k')} \bigotimes_{\mathbb C} \hat{K}_{n})=
({\mathcal L}_{j}^{q+1})^{(k')} \bigotimes_{\mathbb C} \hat{K}_{n}$ we deduce
${\mathcal M}_{j+1} \circ \varphi = {\mathcal M}_{j+1}$ and
$\varphi^{*} \tilde{\mathcal L}_{j+1}^{q} \subset \tilde{\mathcal L}_{j+1}^{q+1}$.
These properties lead to $\varphi^{*} {\mathcal L}_{j+1}^{q} \subset {\mathcal L}_{j+1}^{q+1}$
for any $\varphi \in {\mathcal S}(q-p,q)$.

Given $\varphi \in {\mathcal S}(q-p,q)$, we have
\[ \varphi_{*} {\mathcal L} = \varphi_{*}( {\mathcal L}_{1}^{q} + \ldots + {\mathcal L}_{m}^{q}) \subset
{\mathcal L}_{1}^{q+1} + \ldots + {\mathcal L}_{m}^{q+1} = {\mathcal L}. \]
We remind that $\varphi^{-1} \in {\mathcal S}(q-p,q)$ if $\varphi \in {\mathcal S}(q-p,q)$.
Hence the properties $\varphi_{*} {\mathcal L} \subset {\mathcal L}$ and $(\varphi^{-1})_{*} {\mathcal L} \subset {\mathcal L}$
imply $\varphi_{*} {\mathcal L} = {\mathcal L}$ for any $\varphi \in {\mathcal S}(q-p,q)$.
\end{proof}
Consider the Lie algebra ${\mathfrak g}$ of the pro-algebraic group $\overline{G}^{z}$.
We define ${\mathfrak h} = {\mathfrak g} \cap {\mathcal L}$.
We denote by $\overline{\mathfrak h}^{k}$ the closure of ${\mathfrak h}$ in the Krull topology
(it can be proved that $\overline{\mathfrak h}^{k}$ is equal to ${\mathfrak h}$ but this result
will not be necessary in the following). Remark \ref{rem:krull} implies that
$\overline{\mathfrak h}^{k}$ is a complex Lie subalgebra of ${\mathfrak g}$.
%
\begin{pro}
\label{pro:solvg}
The set $\mathrm{exp}(\overline{\mathfrak h}^{k}): = \{ \mathrm{exp} (X) : X \in \overline{\mathfrak h}^{k} \}$
is a solvable pro-algebraic subgroup of $\overline{G}^{z}$ with Lie algebra $\overline{\mathfrak h}^{k}$.
\end{pro}
\begin{proof}
The Lie correspondence is explicit in the unipotent case: the Lie algebras of unipotent pro-algebraic groups
are the Lie algebras of formal nilpotent vector fields that are closed in the Krull topology \cite{JR:arxivdl2}.
Since ${\mathfrak g}$ and then $\overline{\mathfrak h}^{k}$ consist of nilpotent elements, we obtain
that $\overline{\mathfrak h}^{k}$ is the Lie algebra of a unipotent pro-algebraic group. Such a group is
equal to $\mathrm{exp}(\overline{\mathfrak h}^{k})$ by Lemma \ref{lem:bas}.

Let us show that $\overline{\mathfrak h}^{k}$ is solvable.
Since $(\overline{\mathfrak l}^{k})^{(1)} \subset \overline{{\mathfrak l}^{(1)}}^{k}$ for any Lie subalgebra of
$\hat{\mathfrak X} \cn{n}$, we obtain
$(\overline{\mathfrak h}^{k})^{(j)} \subset \overline{{\mathfrak h}^{(j)}}^{k}$ for any $j \geq 0$.
In particular we get $\ell (\overline{\mathfrak h}^{k}) \leq \ell ({\mathfrak h}) \leq \ell ({\mathcal L}) \leq 2m$
by Proposition \ref{pro:sla}.
Since $\overline{\mathfrak h}^{k}$ is solvable, the group $\mathrm{exp}(\overline{\mathfrak h}^{k})$
is solvable by Lemma \ref{lem:bas}.
\end{proof}
\begin{lem}
\label{lem:nor2}
$\mathrm{exp}(\overline{\mathfrak h}^{k}) \cap \Gamma(q-p)$ is a normal subgroup of $\Gamma(q-p)$
that contains $\Gamma (m+1)$.
\end{lem}
\begin{proof}
Let $\varphi \in {\mathcal S}(q-p,q)$.
We have $\varphi_{*} {\mathcal L}={\mathcal L}$ by Lemma \ref{lem:nor1}.
Moreover $\varphi \in G$ implies $\varphi  \overline{G}^{z} \varphi^{-1}= \overline{G}^{z}$ and then
$\varphi_{*} {\mathfrak g}={\mathfrak g}$.
We deduce $\varphi_{*} {\mathfrak h}={\mathfrak h}$ and then
$\varphi_{*} \overline{\mathfrak h}^{k}=\overline{\mathfrak h}^{k}$.

The infinitesimal generator $\log \eta$   belongs to
${\mathcal L}$ and then to $\overline{\mathfrak h}^{k}$ for any $\eta \in {\mathcal S}(q)$.
In particular $\eta$ belongs to $\mathrm{exp}(\overline{\mathfrak h}^{k})$.

Fix $j \geq q$. We claim that that $\eta \in \mathrm{exp}(\overline{\mathfrak h}^{k})$ and
$\varphi_{*} \overline{\mathfrak h}^{k} = \overline{\mathfrak h}^{k}$ for all
$\eta \in {\mathcal S}(j)$ and
$\varphi \in {\mathcal S}(j-p,j)$.
The proof is by induction on $j$. We already proved the result for $j=q$.
Let us show that if it holds for $j$ then so it does for $j+1$.
Fix $\phi \in {\mathcal S}(j)$ and
$\varphi \in {\mathcal S}(j-p,j)$.
Since $\log \phi \in \overline{\mathfrak h}^{k}$ and
$\varphi_{*} \overline{\mathfrak h}^{k}=\overline{\mathfrak h}^{k}$, we deduce
$\varphi \circ \phi \circ \varphi^{-1} \in \mathrm{exp}(\overline{\mathfrak h}^{k})$.
The commutators
\[ [\varphi, \phi] = (\varphi \circ \phi \circ \varphi^{-1}) \circ \phi^{-1} \ \mathrm{and} \
[\phi, \varphi]= \phi \circ (\varphi \circ \phi^{-1} \circ \varphi^{-1}) \]
are compositions of elements of $\mathrm{exp}(\overline{\mathfrak h}^{k})$ and then belong to
$\mathrm{exp}(\overline{\mathfrak h}^{k})$ by Proposition \ref{pro:solvg}.
By varying $\varphi \in {\mathcal S}(j-p,j)$ and $\phi \in {\mathcal S}(j)$
we obtain that ${\mathcal S}(j+1) \subset \mathrm{exp}(\overline{\mathfrak h}^{k}) \cap \Gamma(q-p)$.
Moreover $\eta_{*} \overline{\mathfrak h}^{k}$ is equal to $\overline{\mathfrak h}^{k}$ for any $\eta \in {\mathcal S}(j+1)$
since $\eta \in \mathrm{exp}(\overline{\mathfrak h}^{k})$.

We proved $\varphi_{*} \overline{\mathfrak h}^{k}=\overline{\mathfrak h}^{k}$ for any
$\varphi \in \cup_{j \geq q-p} {\mathcal S}(j)$. Thus
$\varphi$ normalizes $\overline{\mathfrak h}^{k}$ for any $\varphi \in \langle \cup_{j \geq q-p} {\mathcal S}(j) \rangle$.
Since $\Gamma (q-p)=\langle \cup_{j \geq q-p} {\mathcal S}(j) \rangle$, we deduce
that $\varphi$ normalizes $\mathrm{exp}(\overline{\mathfrak h}^{k})$ for any $\varphi \in \Gamma (q-p)$.
Hence
$\mathrm{exp}(\overline{\mathfrak h}^{k}) \cap \Gamma(q-p)$ is a normal subgroup of $\Gamma(q-p)$.

By construction ${\mathcal S}(m+1,q)$ is contained in
$\mathrm{exp}(\overline{\mathfrak h}^{k}) \cap \Gamma(q-p)$.
We proved $\cup_{j \geq q} {\mathcal S}(j) \subset \mathrm{exp}(\overline{\mathfrak h}^{k}) \cap \Gamma(q-p)$.
Since $\mathrm{exp}(\overline{\mathfrak h}^{k}) \cap \Gamma(q-p)$ is a group,
the group $\Gamma (m+1) = \langle \cup_{j > m} {\mathcal S}(j) \rangle$
is contained in $\mathrm{exp}(\overline{\mathfrak h}^{k}) \cap \Gamma(q-p)$.
\end{proof}
The next proposition completes the proof of the inductive step and as a consequence the proof of
Theorem \ref{teo:psis}.
\begin{pro}
$\Gamma (m)$ is solvable.
\end{pro}
\begin{proof}
Since $q-p \leq m$,
the group $\Gamma (m) \cap \mathrm{exp}(\overline{\mathfrak h}^{k})$ is normal in $\Gamma (m)$
by Lemma \ref{lem:nor2}.
We define the group $H= \Gamma(m) / (\Gamma (m) \cap \mathrm{exp}(\overline{\mathfrak h}^{k}))$.
The property $\Gamma (m+1) \subset \mathrm{exp}(\overline{\mathfrak h}^{k})$ (Lemma \ref{lem:nor2})
implies that $H$ is generated by the classes of elements of ${\mathcal S}(m)$.
Since the commutator of elements of ${\mathcal S}(m)$ belongs to ${\mathcal S}(m+1)$, the group
$H$ is abelian. The group $\Gamma(m)^{(1)}$ is contained in $\mathrm{exp}(\overline{\mathfrak h}^{k})$
and hence $\Gamma (m)$ is solvable by Proposition \ref{pro:solvg}.
\end{proof}
\subsection{Consequences}
We show that, given a pseudogroup induced by a
non-solvable group of unipotent local
diffeomorphisms, all points are recurrent outside of a measure zero set.
This result was proved in dimension $2$
by Rebelo and Reis \cite{RR:arxiv}.
\begin{pro}
\label{pro:recuru}
Let $G \subset \mathrm{Diff}_{u}({\mathbb C}^{n},0)$ be a non-solvable group.
Then there exist $r>0$ and a sequence $(f_{j})_{j \geq 1}$ in the
pseudogroup ${\mathcal P}$ generated by $\{ f_{|{\mathbb B}_{r}^{n}} : f \in G \}$
such that $f_{j}$ is defined in ${\mathbb B}_{r/2}^{n}$,
$(f_{j})_{|{\mathbb B}_{r/2}^{n}} \not \equiv Id$ for any $j \geq 1$
and $\lim_{j \to \infty} ||f_{j} - Id||_{r/2}=0$. In particular all the points in
${\mathbb B}_{r/2}^{n}$ except at most a countable union of proper analytic sets are
recurrent for the action of ${\mathcal P}$.
\end{pro}
\begin{proof}
A unipotent subgroup $H$ of $\mathrm{Diff}({\mathbb C}^{n},0)$
is solvable if and only if
$\ell (H) \leq 2n-1$ \cite{JR:arxivdl}[Theorem 4].
We have $G^{(2n-1)} = \cup H^{(2n-1)}$ where the union is considered over the finitely
generated subgroups of $G$. Hence up to replace $G$ with one of its subgroups we can
suppose that $G$ is finitely generated.  Fix $p \in {\mathbb N}$ such that
Theorem \ref{teo:psis} holds. Hence
$G$ is non-$p$-pseudo-solvable by Theorem \ref{teo:psis}.
The remainder of the proof is an immediate consequence of
Proposition \ref{pro:estup}.
\end{proof}
\section{Linear groups}
Let us deal with the cases in Theorem \ref{teo:rec} besides the first one,
that was already settled in Proposition \ref{pro:recuru}.
The other cases are of linear type, indeed we will construct
free subgroups of $j^{1} G$ with free generators arbitrarily
close to $Id$ by using the Tits alternative \cite{Tits}.
In this way we obtain free subgroups of $G$; they are clearly non-$p$-pseudo-solvable
for any $p \in {\mathbb N}$ (cf. Lemma \ref{lem:free})
and hence we can apply Proposition \ref{pro:estl} to obtain recurrent
points.

The linear part $j^{1} G$ of a subgroup $G$ of $\diff{}{n}$ satisfies the Tits alternative,
i.e. either $j^{1} G$ is virtually solvable or it contains a non-abelian
free group. A more precise result by Breuillard and Gelander is the topological
Tits alternative: a subgroup of $\mathrm{GL}(n,{\mathbb C})$ either contains an
open solvable subgroup or a non-abelian dense free subgroup \cite{Breuillard-Gelander:Tits}.
We will use this kind of ideas in  sections \ref{sec:gwnhe}, \ref{sec:gwnscc} and \ref{sec:indg} to
obtain free subgroups of linear groups. At that point we will apply these results to the study of
groups of local diffeomorphisms to show Theorem \ref{teo:rec}.
\begin{defi}
Let $H$ be a subgroup of $\mathrm{GL}(n,{\mathbb C})$.
We denote by $\overline{H}$ the topological closure of $H$.
It is well-known that $\overline{H}$ is a real Lie group (cf.  \cite{Hall-liegroup}[p. 52, Corollary 2.33]).
We denote by $\overline{H}_{0}$ its connected component of the
identity.
\end{defi}
\subsection{Groups without hyperbolic elements}
\label{sec:gwnhe}
In this section we focus on Case (2) of Theorem \ref{teo:rec}.
\begin{defi}
We say that an element $A$ of $\mathrm{GL}(n,{\mathbb C})$ is
hyperbolic if $\mathrm{spec}(A) \not \subset {\mathbb S}^{1}$.
We say that $\phi \in \mathrm{Diff} ({\mathbb C}^{n},0)$ is hyperbolic
if $D_{0} \phi$ is hyperbolic.
\end{defi}
The main result of this section is next theorem
\begin{teo}
\label{teo:nonhyp}
Let $H$ be a subgroup of $\mathrm{GL}(n,{\mathbb C})$.
Suppose that $H$ is non-virtually solvable and does not contain hyperbolic
elements. Then given any neighborhood $V$ of $Id$ in $\mathrm{GL}(n,{\mathbb C})$
there exists $A, B \in H \cap V$ such that
$A$ and $B$ are free generators of the free group $\langle A, B \rangle$.
In particular the group $\overline{H}_{0}$ is non-virtually solvable.
\end{teo}
Let $H$ be a group satisfying the hypotheses of Theorem \ref{teo:nonhyp}.
It is easier to prove Theorem \ref{teo:nonhyp} if $H$ is irreducible since then
we can use Burnside's theorem. Anyway we associate to $H$ a sequence of
irreducible representations that are useful to show Theorem \ref{teo:nonhyp}.

Consider a sequence
\begin{equation}
\label{equ:seqi}
 \{0\} = V_{0} \subsetneq  V_{1} \subsetneq  \ldots \subsetneq  V_{r} = {\mathbb C}^{n}
\end{equation}
where $V_{j}$ is $H$-invariant for any $0 \leq j \leq r$.
An example is provided by $r=1$, $V_{0}=\{0\}$ and $V_{1}={\mathbb C}^{n}$.
The group $H$ acts on the vector space $V_{j+1}/V_{j}$ for $0 \leq j < r$.
If the action is not irreducible then there exists a $H$-invariant vector space
$V_{j+1/2}$ such that $V_{j} \subsetneq V_{j+1/2} \subsetneq  V_{j+1}$, hence
we can refine the sequence (\ref{equ:seqi}) by introducing the subspace $V_{j+1/2}$.
Since we can not refine indefinitely, there exists a sequence (\ref{equ:seqi})
of $H$-invariant subspaces such that the action of $H$ on $V_{j+1}/V_{j}$ is irreducible
for any $0 \leq j \leq r$. We denote $d_{j}= \dim V_{j}$ and $c_{j}= \dim (V_{j}/V_{j-1})$.
We define $H_{j}$ as the group induced by $H$ on $V_{j}/V_{j-1}$ for any $1 \leq j \leq r$.
The group $H_{j}$ is irreducible by construction and it does not contain hyperbolic elements
by hypothesis. Moreover we have
\begin{lem}
The group $H_{j}$ is  relatively compact
in $\mathrm{GL}(V_{j}/V_{j-1})$ for any $1 \leq j \leq r$.
\end{lem}
\begin{proof}
Since $H_{j}$ is an irreducible subgroup of $\mathrm{GL}(V_{j}/V_{j-1})$,
Burnside's theorem implies that there exist $c_{j}^{2}$ ${\mathbb C}$-linearly independent
elements $g_{1}, \ldots, g_{c_{j}^{2}}$ of $H_{j}$ (cf. \cite{Wehrfritz}[p. 11, Corollary 1.17]).

Consider now the symmetric bilinear form $\alpha : M_{c_{j}}({\mathbb C}) \times M_{c_{j}}({\mathbb C}) \to {\mathbb C}$
defined by $\alpha (f, h) = \mathrm{trace} (f h)$ where $M_{c_{j}}({\mathbb C})$ is the
vector space of $c_{j} \times c_{j}$ complex matrices.
The bilinear form $\alpha$ is non-degenerate;
indeed given $f \in M_{c_{j}}({\mathbb C}) \setminus \{0\}$ there exists
$v \in {\mathbb C}^{c_{j}}$ such that $f(v) \neq 0$.
Consider a base $\{w_{1}, \ldots, w_{c_{j}} \}$ of ${\mathbb C}^{c_{j}}$ such that
$w_{1}=f(v)$ and a linear map $h: {\mathbb C}^{c_{j}} \to {\mathbb C}^{c_{j}}$ with $h(w_{1})=v$ and
$h(w_{k})=0$ for any $2 \leq k \leq c_{j}$.
The map $fh$ satisfies $fh(w_{1}) =w_{1}$ and $fh(w_{k})=0$ for any $2 \leq k \leq c_{j}$.
Hence the trace of $fh$ is equal to $1$.
The non-degenerate nature of $\alpha$ implies the existence of a dual basis
$\{ e_{1}, \ldots, e_{c_{j}^{2}} \}$ of $M_{c_{j}}({\mathbb C})$ such that
$\alpha (g_{k}, e_{l}) = \delta_{kl}$ for any $1 \leq k,l \leq c_{j}^{2}$.
Thus we obtain
\[ g = \sum_{k=1}^{c_{j}^{2}} \alpha(g,g_{k}) e_{k}  = \sum_{k=1}^{c_{j}^{2}} \mathrm{trace} (g g_{k}) e_{k} \]
for any $g \in M_{c_{j}}({\mathbb C})$. We deduce
\[ H_{j} \subset \left\{ \sum_{k=1}^{c_{j}^{2}} t_{k} e_{k}: \ t_{k} \in \tau (H_{j}) \right\} \]
where $\tau (H_{j})$ is the set of traces of elements of $H_{j}$.
The lack of hyperbolic elements implies $\tau (H_{j}) \subset \overline{{\mathbb B}_{c_{j}}^{1}}$.
In particular $H_{j}$ is bounded and then relatively compact
in $\mathrm{GL}(V_{j}/V_{j-1})$ for any $1 \leq j \leq r$.
\end{proof}
\begin{lem}
\label{lem:semi}
The group $\overline{H}_{j}$ consists of semisimple (i.e. diagonalizable) non-hyperbolic
transformations for any $1 \leq j \leq r$.
\end{lem}
\begin{proof}
Fix $1 \leq j \leq r$.
Since the spectrum of a matrix varies continuously, we obtain $\mathrm{spec}(A) \subset {\mathbb S}^{1}$
for any $A \in \overline{H}_{j}$. Fix $A \in \overline{H}_{j}$.
Consider the multiplicative Jordan decomposition $A = A_{s} A_{u}$ as the product of a semisimple matrix
$A_{s}$ and a unipotent matrix $A_{u}$ that commute. Since $\mathrm{spec}(A)= \mathrm{spec}(A_{s})$ and
$A_{s}$ is diagonalizable, the group $\langle A_{s} \rangle$ is relatively compact.
Hence $\langle A_{u} \rangle$ is relatively compact.
Let us show $A_{u} = Id$ by contradiction. Otherwise
the Jordan normal form theorem implies the existence of
linearly independent vectors $v,w$ such that $A_{u} v = v+w$ and $A_{u} w = w$.
Since $A_{u}^{j} v = v +j w$ for $j \in {\mathbb Z}$ the group $\langle A_{u} \rangle$
is non-relatively compact and we obtain a contradiction.
\end{proof}
Consider a free subgroup $\langle a,b \rangle$ on two generators of $\mathrm{GL}(n,{\mathbb C})$
with $a,b$ close to $Id$. The next result has two functions: namely showing that
$\langle a,b \rangle$ is non-$0$-pseudo-solvable (and then non-$p$-pseudo-solvable for any $p \in {\mathbb N} \cup \{0\}$)
and finding non-abelian free subgroups of $\langle a,b \rangle$ with free generators even closer to $Id$.
\begin{lem}
\label{lem:free}
Let $H$ be a free group on $\{a,b\}$ and
${\mathcal S}  = \{a,b,a^{-1},b^{-1}\}$.
Then ${\mathcal S}_{0} (2k)$ contains $2$ elements that are free generators of
a free group for any $k \geq 0$. Moreover the set $\cup_{j \geq 0} {\mathcal S}_{0} (j)$ is infinite.
\end{lem}
\begin{proof}
Fix $k \geq 0$.
We want to prove that in ${\mathcal S}_{0} (2k)$ there is a word $\alpha_{f,k}$
that in reduced form  has length $4^{2k}$ and $f$ as first and also as last letter for any $f \in {\mathcal S}$.
The result is obvious for $k=0$. Let us show that if it holds for $k$ then so it does for $k+1$. We have that
\[ [\alpha_{a,k}, \alpha_{b^{-1},k}] = a \ldots b \ \mathrm{and} \  [\alpha_{a^{-1},k}, \alpha_{b,k}]  =a^{-1} \ldots b^{-1} \]
are words of length $4^{2k+1}$ in ${\mathcal S}_{0} (2(k+1))$.  We can define
\[ \alpha_{a,k+1} =[[\alpha_{a,k}, \alpha_{b^{-1},k}],[\alpha_{a^{-1},k}, \alpha_{b,k}] ]=
a \ldots b a^{-1} \ldots b^{-1} b^{-1} \ldots a^{-1} b \ldots a . \]
The word $\alpha_{a,k+1}$ belongs to ${\mathcal S}_{0} (2(k+1))$ and has length $4^{2(k+1)}$.
The words $\alpha_{a^{-1},k+1}$,
$\alpha_{b,k+1}$ and $\alpha_{b^{-1},k+1}$ are defined analogously.

It is clear that $\alpha_{a,k}$ and $\alpha_{b,k}$ are free generators of a free group on two elements
for any $k \geq 0$.
Moreover  $\cup_{j \geq 0} {\mathcal S}_{0} (j)$ is infinite since it contains
reduced words of arbitrarily high length.
\end{proof}
The next result is a simple exercise; it is a consequence of the compactness of $({\mathbb S}^{1})^{l}$.
\begin{lem}
\label{lem:tri}
Consider $\lambda_{1} , \ldots, \lambda_{l} \in {\mathbb S}^{1}$. Then there exists an increasing
sequence $(n_{k})_{k \geq 1}$ of natural numbers such that
$\lim_{k \to \infty} \lambda_{j}^{n_{k}} = 1$ for any $1 \leq j \leq l$.
\end{lem}
\begin{proof}[end of the proof of Theorem \ref{teo:nonhyp}]
Let us remind the reader that a subgroup of $\mathrm{GL}(n,{\mathbb C})$
is either virtually solvable or it contains a non-abelian free group by
the Tits alternative \cite{Tits}. Therefore
there exist $A,B \in H$ that
are free generators of a free group on two elements.
Lemmas \ref{lem:semi} and \ref{lem:tri} imply the existence of a increasing sequence
$(n_{k})_{k \geq 1}$ of natural numbers
such that $A_{|V_{j}/V_{j-1}}^{n_{k}}$ and $B_{|V_{j}/V_{j-1}}^{n_{k}}$ tend to the identity map
when $k \to \infty$ for any $1 \leq j \leq r$.

Given a matrix $D$ we denote by $D_{I,J}$ the minor that corresponds to the rows with index $I$ and
the columns with index $J$.
We define the diagonal minors
$D_{j} = D_{\{d_{j-1}+1, \ldots, d_{j}\}, \{d_{j-1}+1, \ldots, d_{j}\}}$ for any $1 \leq j \leq r$.
Consider a base ${\mathcal B}= \{v_{1}, \ldots, v_{n} \}$ of ${\mathbb C}^{n}$ such that
$\{v_{1}, \ldots, v_{d_{l}} \}$ is a base of $V_{l}$ for any $1 \leq l \leq r$.
We identify $A$ and $B$ with its matrices in the basis ${\mathcal B}$.
They are block upper triangular matrices.
The condition $\lim_{k \to \infty} A_{|V_{j}/V_{j-1}}^{n_{k}} =Id$ implies
$\lim _{k \to \infty} (A^{n_{k}})_{ j} =Id$. We define the matrix
$C^{s}$ for $s \in {\mathbb R}^{+}$ such that
$C_{j,j}^{s} = s^{l}$ if $d_{l} < j \leq d_{l+1}$ and $C_{j,k}^{s}=0$ if $j \neq k$.
We obtain $((C^{s})^{-1} A^{n_{k}} C^{s})_{j} = A_{j}^{n_{k}}$ and
$((C^{s})^{-1} B^{n_{k}} C^{s})_{j} = B_{j}^{n_{k}}$ for all $s>0$, $1 \leq j \leq r$ and $k \geq 1$
whereas all the coefficients of the matrix $(C^{s})^{-1} B^{n_{k}} C^{s}$ outside of the diagonal
minors tend to $0$ when $s \to 0$.
Hence by considering $k \in {\mathbb N}$ big enough and then $s >0$ small enough
we obtain matrices $\tilde{A} = (C^{s})^{-1} A^{n_{k}} C^{s}$ and $\tilde{B} = (C^{s})^{-1} B^{n_{k}} C^{s}$
arbitrarily close to the identity. Moreover $\tilde{A}$ and $\tilde{B}$ are free generators of the group
$\langle \tilde{A},\tilde{B} \rangle$.
Let $W$ be the neighborhood of $Id$ provided by
Remark \ref{rem:lie} for $\mathrm{GL}(n,{\mathbb C})$.
We can suppose $\tilde{A}, \tilde{B}, \tilde{A}^{-1} \tilde{B}^{-1} \in W$.

We define ${\mathcal S} = \{\tilde{A}, \tilde{B}, \tilde{A}^{-1} \tilde{B}^{-1} \}$. The elements in
${\mathcal S}_{0} (k)$ are contained in $\{ D \in \mathrm{GL}(n,{\mathbb C}): ||D-Id|| < \epsilon / 2^{2^{k}-1} \}$
(cf. Remark \ref{rem:lie}).
Lemma \ref{lem:free} implies the existence of free groups on two elements of
$(C^{s})^{-1} H C^{s}$ arbitrarily close to $Id$ .
Therefore there exist free groups on two elements of $H$ arbitrarily close to $Id$.
\end{proof}
\subsection{Groups with non-solvable connected component of $Id$}
\label{sec:gwnscc}
In our quest for free groups we consider now subgroups $H$ of $\mathrm{GL}(n,{\mathbb C})$
such that $\overline{H}_{0}$ is non-solvable.
\begin{defi}
We say that a connected Lie group $T$ is topologically perfect if $T= \overline{[T,T]}$.
\end{defi}
Next we provide the statements of two theorems by Breuillard and Gelander that we
use to find free groups in a neighborhood of $Id$. In particular Theorem \ref{teo:Br-Ge-free}
is particularly interesting since it localizes the
free generators of the free group
\begin{teo}
\cite{Breuillard-Gelander:dense}
\label{teo:Br-Ge-per}
Let $T$ be a topologically perfect real Lie group. There exists a neighborhood $\Omega$ of $Id$,
where $\mathrm{exp}^{-1}$ is a diffeomorphism from $\Omega$ to a neighborhood of
$0$ in $L(T)$, such that given $f_{1}, \ldots, f_{m} \in \Omega$, the group
$\langle f_{1}, \ldots, f_{m} \rangle$ is dense in $T$ if
$\log f_{1}, \ldots, \log f_{m}$ generates $L(T)$.
The neighborhood $\Omega$ can be considered arbitrarily small.
\end{teo}
\begin{teo}
\cite{Breuillard-Gelander:Tits}
\label{teo:Br-Ge-free}
Let $k$ be a characteristic $0$ local field.
Let $T= \langle f_{1},\ldots,f_{m} \rangle$ be a subgroup of $\mathrm{GL}(n,k)$
that contains no open solvable subgroup.
Then given any neighborhood $\Omega$ of $Id$
in $T$, there exist $g_{j} \in \Omega f_{j} \Omega$ for any $1 \leq j \leq m$ such that
$g_{1}, \ldots, g_{m}$ are free generators of a free dense subgroup of $T$.
\end{teo}
The algebraic properties of a linear group and its topological closure coincide.
Next lemma provides a particular case of the previous principle.
\begin{lem}
\label{lem:solc}
Given a subgroup $T$ of $\mathrm{GL}(n,{\mathbb C})$ we have
$T^{(j)} \subset \overline{T}^{(j)} \subset \overline{T^{(j)}}$ for any $j \geq 0$.
In particular we deduce  that $\ell (T)$ is equal to $\ell (\overline{T})$ and that $T$ is solvable if and only if
$\overline{T}$ is solvable.
\end{lem}
\begin{proof}
We have $\overline{T}^{(1)} \subset \overline{T^{(1)}}$. By recurrence
we obtain $T^{(j)} \subset \overline{T}^{(j)} \subset \overline{T^{(j)}}$ for any $j \geq 0$.
In particular we obtain $\ell (T) = \ell (\overline{T})$.
\end{proof}
Next we introduce and prove the main result of this section.
\begin{teo}
\label{teo:nsolcc}
Let $H$ be a subgroup of $\mathrm{GL}(n,{\mathbb C})$ such that $\overline{H}_{0}$ is non-solvable.
Then given any neighborhood $U$ of $Id$ in $H$, there exists free generators $f,g \in  H \cap U$ of a free
group on two elements.
\end{teo}
\begin{proof}
We denote $H_{0}^{\circ} =  \overline{H}_{0}$ and
$H_{j+1}^{\circ} = \overline{[H_{j}^{\circ}, H_{j}^{\circ}]}$ for $j \geq 0$.
The construction of $H_{j}^{\circ}$ implies
$\overline{(\overline{H}_{0} \cap H)^{(j)}} \subset H_{j}^{\circ}$ for any $j \geq 0$.
Let us show that $H_{j}^{\circ}$ is contained in $\overline{(\overline{H}_{0} \cap H)^{(j)}}$ for
$j \geq 0$. It is obvious for $j=0$. Suppose it holds for $j \geq 0$.
We have
\[ H_{j+1}^{\circ} \subset  \overline{[\overline{(\overline{H}_{0} \cap H)^{(j)}},\overline{(\overline{H}_{0} \cap H)^{(j)}}]}  \subset
\overline{\overline{(\overline{H}_{0} \cap H)^{(j+1)}}} = \overline{(\overline{H}_{0} \cap H)^{(j+1)}} \]
by Lemma \ref{lem:solc}. We obtain $H_{j}^{\circ} \subset \overline{(\overline{H}_{0} \cap H)^{(j)}}$ for any $j \geq 0$
by induction on $j$. Thus
$H_{j}^{\circ} = \overline{(\overline{H}_{0} \cap H)^{(j)}}$ holds for any $j \geq 0$.

Since $\overline{H}_{0} \cap H$ is dense in $\overline{H}_{0}$,
the group $\overline{H}_{0} \cap H$ is non-solvable by Lemma \ref{lem:solc}.
A simple consequence is that $H_{j}^{\circ}$ is never
the trivial group for any $j \geq 0$.

The group $H_{0}^{\circ}$ is connected. Moreover if $H_{j}^{\circ}$ is connected then
$[H_{j}^{\circ}, H_{j}^{\circ}]$ and $H_{j+1}^{\circ}$ are connected too for $j \geq 0$.
We deduce that $H_{j}^{\circ}$ is connected for any $j \geq 0$.
We obtain a decreasing sequence $H_{0}^{\circ} \supset H_{1}^{\circ} \supset H_{2}^{\circ} \supset \ldots$ of
connected Lie groups. Given $j \geq 0$ either we have
$H_{j}^{\circ} = H_{j+1}^{\circ}$ or $\dim H_{j}^{\circ} < \dim H_{j+1}^{\circ}$.
Hence there exists $j_{0} \geq 0$ such that $\{Id \} \neq H_{j_{0}}^{\circ} = H_{j_{0}+1}^{\circ}$.
By replacing $H$ with ${(\overline{H}_{0} \cap H)}^{(j_{0})}$ we can suppose that $H$ is a non-solvable group
such that $\overline{H}$ is topologically perfect.

Let $\Omega$ be a neighborhood of $Id$ for $\overline{H}$ as provided by
Theorem \ref{teo:Br-Ge-per}. We can suppose $\Omega \subset U$.
It is clear that their exist $f_{1}, \hdots, f_{m} \in \Omega \cap H$ such that
$\overline{H}= \overline{J}$ where $J = \langle f_{1}, \hdots, f_{m} \rangle$. Moreover
this can be achieved for $m= \dim \overline{H}$.

Let us show that $J$ does not contain any open solvable subgroup.
Otherwise there exists a neighborhood $V$ of $Id$ in $\mathrm{GL}(n,{\mathbb C})$
such that $\langle J \cap V \rangle$ is solvable.
The group $\overline{\langle J \cap V \rangle}$ is solvable by Lemma \ref{lem:solc}.
Since such a group contains $\langle \overline{J \cap V} \rangle$, we deduce that
$\langle \overline{H} \cap V \rangle$ is solvable. A connected Lie group is generated
by any of its neighborhoods of the identity, hence $\overline{H}$ is solvable and we obtain
a contradiction.

Theorem \ref{teo:Br-Ge-free} implies that the
there exist $g_{1}, \ldots, g_{m} \in J$ arbitrarily
close to $f_{1}, \ldots, f_{m}$ respectively such that
$g_{1}, \ldots, g_{m}$ are free generators of a free dense subgroup of $J$.
We can suppose $g_{1}, \ldots, g_{m} \in U$.
It is clear that $m \geq 2$ since $\langle g_{1}, \hdots, g_{m} \rangle$ is non-solvable
by Lemma \ref{lem:solc}.
\end{proof}
%
%
%
%
%
\subsection{Irreducible and non-discrete groups}
\label{sec:indg}
We consider conditions on a linear group $H$ forcing $\overline{H}_{0}$
to be non-solvable. Indeed we require $H$ and all its finite index subgroups
to be irreducible and the group induced by $H$ in $\mathrm{PGL}(n,{\mathbb C})$
to be discrete (cf. Theorem \ref{teo:nvrnd}).
The next result corresponds to one of the cases in Theorem \ref{teo:rec}
and it is also an intermediate result to show Theorem \ref{teo:nvrnd}.
\begin{teo}
\label{teo:ccind}
Let $H$ be a subgroup of $\mathrm{GL}(n,{\mathbb C})$ that is not virtually reducible.
Suppose further that $\overline{H}_{0}$ is not contained in ${\mathbb C}^{*} Id$.
Then given any neighborhood $U$ of $Id$ in $H$, there exist $f,g  \in H \cap U$ such that
$f,g$ are free generators of the free group $\langle f,g \rangle$.
\end{teo}
\begin{proof}
It suffices to show that the hypotheses imply that $\overline{H}_{0}$ is non-solvable and then
to apply Theorem \ref{teo:nsolcc}.

Suppose $\overline{H}_{0}$ is solvable. We define
\[ S = \{ v \in {\mathbb C}^{n} : \exists \lambda_{A,v} \in {\mathbb C} \ \mathrm{such} \ \mathrm{that} \
A v = \lambda_{A,v} v \ \ \forall A \in \overline{H}_{0} \} \]
of vectors that are eigenvectors for all transformations in $\overline{H}_{0}$.
The set $S$ is a finite union $V_{1}, \ldots, V_{m}$ such that $V_{1} + \ldots + V_{m}$
is a direct sum. Moreover $A_{|V_{j}}$ is a multiple of the identity map for
all $A \in \overline{H}_{0}$ and $1 \leq j \leq m$.
Let us remark that $S \neq \{0\}$ since all elements of $\overline{H}_{0}$ have a common
eigenvector by Lie-Kolchin theorem (cf. \cite{Serre.Lie}[p. 38, Theorem 5.1*]).
The group $\overline{H}_{0}$ is normal in $\overline{H}$, hence
the elements of $\overline{H}$ permute the subspaces $V_{1}, \ldots, V_{m}$.
There exists a finite index normal subgroup $J$ of $H$ such that
$A (V_{j})=V_{j}$ for any $1 \leq j \leq m$.
Notice that the situation $m=1$ and $V_{1}= {\mathbb C}^{n}$ is impossible since then
$\overline{H}_{0} \subset {\mathbb C}^{*} Id$.
We deduce that $J$ is reducible and then $H$ is virtually reducible, obtaining
a contradiction.
\end{proof}
\begin{teo}
\label{teo:nvrnd}
Let $H$ be a subgroup of $\mathrm{GL}(n,{\mathbb C})$ that is not virtually reducible.
Suppose that the group induced by ${H}$ in $\mathrm{PGL}(n,{\mathbb C})$ is non-discrete.
Then given any neighborhood $U$ of $Id$ in $\mathrm{GL}(n,{\mathbb C})$, there exist
$A,B  \in H \cap U$ such that $A,B$ are free generators of the free group $\langle A,B \rangle$.
In particular $\overline{H}_{0}$ is non-virtually solvable.
\end{teo}
\begin{proof}
Consider the natural map $\Lambda : \mathrm{GL}(n,{\mathbb C}) \to \mathrm{PGL}(n,{\mathbb C})$.
We denote $J= \Lambda^{-1} (\Lambda (H))$.
The group $J$ is not virtually reducible.
Even if a priori $H$ can be strictly contained in $J$,
its derived groups $H^{(1)}$ and $J^{(1)}$ coincide.

Since $\Lambda ({H})$ is non-discrete, the group
$\overline{J}$ has elements arbitrarily close to $Id$ that
do not belong to ${\mathbb C}^{*} Id$.
We deduce that $\overline{J}_{0}$ is not contained in ${\mathbb C}^{*} Id$.
There are elements $C,D \in J \cap U $ that are free generators of a free
group on two elements by Theorem \ref{teo:ccind}. We can even suppose
$C,D,C^{-1},D^{-1} \in W$ where $W$ is the neighbourhood of $Id$ in $\mathrm{GL}(n,{\mathbb C})$
provided by Remark \ref{rem:lie}.
We define ${\mathcal S} =\{C,D,C^{-1},D^{-1}\}$. The set
${\mathcal S}_{0} (k)$ is contained in $\{ E \in \mathrm{GL}(n,{\mathbb C}): ||E-Id|| < \epsilon / 2^{2^{k}-1} \}$
(cf. Remark \ref{rem:lie}) for any $k \geq 0$. By considering $k > 0$ big enough
we obtain free groups on two elements whose free generators are contained in
$U \cap {\mathcal S}_{0} (2k) \cap J$
(Lemma \ref{lem:free}).  Since ${\mathcal S}_{0} (2k) \subset J^{(1)} = H^{(1)} \subset H$ we are done.
\end{proof}
We have already all the ingredients to prove the main theorem.
\begin{proof}[Proof of Theorem \ref{teo:rec}]
The result is a consequence of Proposition \ref{pro:recuru} in Case (1).
Hence we can always suppose that $G^{1}:=G \cap \mathrm{Diff}_{1} ({\mathbb C}^{n},0)$
is solvable.

Let us show that in the remaining cases, given $\epsilon >0$ there exist
$A,B \in j^{1} G$ such that $||A-Id||_{1} < \epsilon$, $||B-Id||_{1} < \epsilon$ and
$A,B$ are free generators of a free group on two elements.

Let us consider Case (2).
It suffices to show that $j^{1} G$ is non-virtually solvable by
Theorem \ref{teo:nonhyp}. Suppose $j^{1} G$ is virtually solvable.
The natural morphism $G/G^{1} \to j^{1} G$
is an isomorphism of groups.  Hence $G/G^{1}$ is virtually solvable and there
exists a finite index normal solvable subgroup $H/G^{1}$ of $G/G^{1}$.
Since $(G/G^{1})/(H/G^{1})$ is isomorphic to $G/H$, we deduce that
$H$ is a finite index normal subgroup of $G$.
The solvability of $H/G^{1}$ and $G^{1}$ implies that $H$ is solvable.
Therefore $G$ is virtually solvable and we obtain a contradiction.

The result for Cases (3), (4) and (5) is a consequence of
of applying Theorems \ref{teo:nsolcc}, \ref{teo:ccind} and \ref{teo:nvrnd}
respectively to $j^{1} G$.

Consider $\delta >0$ with $\delta < 1/8$.
There exist $\phi, \eta \in G$ such that
$||D_{0} \phi^{\pm 1} - Id||_{1} < \delta/8$, $||D_{0} \eta^{\pm 1} - Id||_{1} < \delta/8$
and $D_{0} \phi$, $D_{0} \eta$ are free generators of a free
group on two elements by the previous discussion.
In particular $\phi$ and $\eta$ are free generators of a free group
on two elements.
We denote ${\mathcal S} = \{\phi, \phi^{-1}, \eta, \eta^{-1} \}$.

The group $\langle \phi, \eta \rangle$ is not $0$-pseudo-solvable for
${\mathcal S}$ by Lemma \ref{lem:free}.
The proof is completed by applying Proposition \ref{pro:estl}.
\end{proof}
\begin{rem}
Conditions (2), (4) and (5) of
Theorem \ref{teo:rec} imply condition (3).
This is a consequence of Theorems \ref{teo:nonhyp}, \ref{teo:ccind} and \ref{teo:nvrnd}
for Cases (2), (4) and (5) respectively.
\end{rem}
\subsection{Groups with hyperbolic elements}
\label{sec:gwhe}
We show Theorem \ref{teo:lcoivs} in this section.
It is natural to try to restrict our study to subgroups $G$ of $\diff{}{n}$
such that $j^{1} G$ is as simple as possible. For instance it is
useful to consider linear groups $j^{1} G$ that share Zariski-closure
with all its finite index normal subgroups. We will obtain such reduction
by using the following proposition.
\begin{pro}
\label{pro:irr}
Let $H$ be a non-virtually solvable subgroup of $\mathrm{GL}(n,{\mathbb C})$.
There exists a subgroup $J$ of $H$ such that
\begin{itemize}
\item $J$ is non-virtually solvable.
\item $\overline{L}^{z} = \overline{J}^{z}$ for any non-virtually solvable subgroup $L$ of $J$.
\item $\overline{J}^{z}$ is a Zariski-connected irreducible algebraic group.
\item $J$ is the derived group of a subgroup $K$ of $H$ satisfying
the first two properties (and then $\overline{K}^{z} = \overline{J}^{z}$).
\end{itemize}
\end{pro}
The proposition is based in a simple idea:  among the Zariski-closures of non-virtually solvable subgroups
of a  non-virtually solvable subgroup $H$ of $\mathrm{GL}(n,{\mathbb C})$ there
are always minimal subgroups.
\begin{proof}
Let us show first that there exists a group $K$ satisfying the first two properties by contradiction.
Otherwise there exists an infinite  decreasing sequence
$H= H_{0} \supset H_{1} \supset H_{2} \supset \ldots$ of non-virtually solvable subgroups of $H$ such that
\begin{equation}
\label{equ:gl}
\overline{H_{0}}^{z}  \supsetneq \overline{H_{1}}^{z} \supsetneq \overline{H_{2}}^{z}  \supsetneq \ldots
\end{equation}
Given $j \geq 0$ we have either $\dim \overline{H_{j+1}}^{z} < \overline{H_{j}}^{z}$ or
$\overline{H_{j+1}}^{z}$ has fewer connected components than $\overline{H_{j}}^{z}$.
Since the number of connected components of an algebraic group is finite, we deduce that the sequence
(\ref{equ:gl}) does not exist. We obtain a contradiction.

We claim that $\overline{K}^{z}$ is connected. We define
$K' = (\overline{K}^{z})_{0} \cap K$. The group $K'$ is
a finite index normal subgroup of $K$ and hence
$K'$ is not virtually solvable. Moreover we have
$\overline{K'}^{z} = (\overline{K}^{z})_{0}$.
Since $\overline{K'}^{z} = \overline{K}^{z}$
by construction of $K$, we deduce $\overline{K}^{z}= (\overline{K}^{z})_{0}$.
Notice that $\overline{K}^{z}$ is a (connected) smooth algebraic set
since $\overline{K}^{z}$ is an algebraic group.
Therefore $\overline{K}^{z}$ is an irreducible algebraic set.

The Tits alternative implies that $K$ contains a free group on two elements
and then $J:=K^{(1)}$ contains non-abelian free groups. In particular $J$ is non-virtually solvable.
We obtain $\overline{J}^{z} = \overline{K}^{z}$ by the second property.
Given a non-virtually solvable subgroup $L$ of $J$ we have $\overline{J}^{z} = \overline{L}^{z}$,
otherwise $K$ does not satisfy the second condition.
\end{proof}
The groups $K$ and $J$ provided by Proposition \ref{pro:irr}
share Zariski-closure with all their finite index subgroups.
Such property is interesting, since for instance the set of $J$-invariant
vector subspaces of ${\mathbb C}^{n}$ depends only on $\overline{J}^{z}$.
\begin{lem}
\label{lem:irrz}
Let $H$ be a subgroup of $\mathrm{GL}(n,{\mathbb C})$ and $V$ a vector subspace of ${\mathbb C}^{n}$.
Then $V$ is $H$-invariant if and only if $V$ is $\overline{H}^{z}$-invariant.
\end{lem}
\begin{proof}
The necessary condition is obvious. Let $V$ be a $H$-invariant subspace.
The set $L:=\{ A \in \mathrm{GL}(n,{\mathbb C}): A(V) = V \}$ is an algebraic group containing $H$.
Hence $L$ contains $\overline{H}^{z}$.
\end{proof}
Let $G$ be a subgroup of $\diff{}{n}$ such that $j^{1} G$ is non-virtually solvable.
In order to show that $G$ does not satisfy the discrete orbits property,
we can suppose that there are elements in $j^{1} G$ whose spectrum is not contained
in ${\mathbb S}^{1}$ by Theorem \ref{teo:rec}. These elements do have either stable or unstable manifolds.
Indeed we will obtain  recurrent point in stable or unstable manifolds of certain hyperbolic
diffeomorphisms.

Let $A \in \mathrm{GL}(n,{\mathbb C})$. The subspaces
\[ V_{A}^{s} = \bigoplus_{\lambda \in \mathrm{spec} (A) \cap {\mathbb B}_{1}^{1}} \mathrm{ker} (A - \lambda Id)^{n},
\ V_{A}^{cu} = \bigoplus_{\lambda \in \mathrm{spec} (A) \setminus {\mathbb B}_{1}^{1}} \mathrm{ker} (A - \lambda Id)^{n} \]
are the stable and the center-unstable manifolds of $A$ respectively.
Consider $\phi \in \mathrm{Diff} ({\mathbb C}^{n},0)$ such that
$D_{0} \phi = A$. The stable manifold theorem (cf. \cite{Ruelle}[p. 27])
implies that there exists a manifold
$W_{\phi}^{s}$ containing the origin such that $\phi (W_{\phi}^{s}) \subset W_{\phi}^{s}$
and $T_{0} W_{\phi}^{s} = V_{A}^{s}$. Moreover $\phi_{|W_{\phi}^{s}}^{p}$ tends uniformly to
$0$ in $W_{\phi}^{s}$ when $p \to \infty$ and
the germ of $W_{\phi}^{s}$ at the origin is
uniquely determined. Analogously we define the unstable manifold $W_{\phi}^{u}$, it is equal to
$W_{\phi^{-1}}^{s}$.
\begin{defi}
Let $\pi : {\mathbb C}^{n} \setminus \{0\} \to  {\bf CP}^{n-1}$ be the map associating to each vector
its class in the projective space. Given a vector subspace $V$ of ${\mathbb C}^{n}$ we
denote $[V] = \pi (V \setminus \{0\})$.
\end{defi}
Next we show a irreducibility type property for actions on stable manifolds.
\begin{pro}
\label{pro:exit2}
Let $H$ be a non-virtually solvable subgroup of $\mathrm{GL}(n,{\mathbb C})$.
Let $K$ and $J$ be the subgroups provided by Proposition \ref{pro:irr}.
Consider $A \in J$. There is no a non-trivial $J$-invariant vector subspace
of $V_{A}^{s}$. In particular
given any $[v] \in  [V_{A}^{s}]$, there exists
$B \in J$ such that $[B v] \not \in [V_{A}^{s}]  \cup [V_{A}^{cu}]$.
\end{pro}
\begin{proof}
We denote $V^{s}=V_{A}^{s}$ and $V^{cu}=V_{A}^{cu}$.
Let $0 \neq V \subset V^{s}$ be a $J$-invariant vector space.
Proposition \ref{pro:irr} implies $\overline{J}^{z} = \overline{K}^{z}$.
Since $V$ is $J$-invariant, it is also $K$-invariant by
Lemma \ref{lem:irrz}.
The property $J = K^{(1)}$
implies $\det (B_{|V})=1$ for any $B \in J$.
Since $V \subset V^{s}$ this contradicts
$\mathrm{spec}{(A_{|V})}  \subset \mathrm{spec}{(A_{|V^{s}})} \subset \{ z \in {\mathbb C} : |z| < 1 \}$.
Thus any $J$-invariant vector subspace of $V^{s}$ is trivial.

Let
\[ J_{s} = \{ B \in J:  [B v] \in [V^{s}]  \} \ \mathrm{and} \
 J_{cu} = \{ B \in J:  [B v] \in [V^{cu}]  \}  . \]
The sets $J_{s}$ and $J_{cu}$ are Zariski-closed in $J$.
We claim $J \not \subset J_{s} \cup J_{cu}$. It suffices to show
$J \not \subset J_{s}$ and $J \not \subset J_{cu}$ since
$\overline{J}^{z}$ is an irreducible algebraic set. The latter property is obvious
since $Id \not \in J_{cu}$.

Let us show $J \not \subset J_{s}$ by contradiction.
Otherwise $B v$ belongs to  $V^{s}$  for any $B \in J$.
Therefore the linear subspace $V$ generated by
$\{ Bv : B \in J\}$ is contained in $V^{s}$.
Moreover by construction $V$ is $J$-invariant, contradicting
the first part of the proof.
\end{proof}
\begin{proof}[Proof of Theorem \ref{teo:lcoivs}]
Suppose that $G$ is not virtually solvable. We have
\[ 0 \to G \cap \mathrm{Diff}_{1} ({\mathbb C}^{n},0) \to G \to j^{1} G \to 0 . \]
Suppose $G \cap \mathrm{Diff}_{1} ({\mathbb C}^{n},0)$ is non-solvable.
Then $G$ does not have the discrete orbits property since
$G \cap \mathrm{Diff}_{1} ({\mathbb C}^{n},0)$ has (lots of) recurrent points
in any neighborhood of the origin by Theorem \ref{teo:rec} (1).
Suppose from now on that
$G \cap \mathrm{Diff}_{1} ({\mathbb C}^{n},0)$ is solvable.
Since $G$ is non-virtually solvable, the group
$H:=j^{1} G$  is non-virtually solvable.

Consider the groups $J$ and $K$ provided by Proposition \ref{pro:irr}.
We define $\tilde{G} = \{ \phi \in G : j^{1} \phi \in J \}$.
Suppose that $\mathrm{spec}(A) \subset {\mathbb S}^{1}$ for any $A \in J$.
Since $\tilde{G}$ is non-virtually solvable, $G$ does not hold the
discrete orbits property by Theorem \ref{teo:rec} (2).
We can suppose that there exists $A \in J$ such that $\mathrm{spec}(A) \cap {\mathbb B}_{1}^{1} \neq \emptyset$.
Consider an element $\phi \in \tilde{G}$ such that $j^{1} \phi =A$.
We obtain $\dim W_{\phi}^{s} \geq 1$.
Let $D^{s}$ be a closed fundamental domain of $\phi$ restricted to $W_{\phi}^{s} \setminus \{0\}$.
We can extend $D^{s}$ to a fundamental domain $M^{s}$ for $\phi$ restricted to a neighborhood of
$D^{s}$.

Let $q \in D^{s}$. The sequence $(\phi^{p}(q))_{p \geq 1}$ tends to $0$ when $p \to \infty$.
Up to consider a subsequence $(p_{k})_{k \geq 1}$ we can suppose that
$([\phi^{p_{k}}(q)])_{k \geq 1}$ converges to a direction $\ell$ in
$[V_{A}^{s}]$ when $k \to \infty$. There exists $B \in J$ such that
$[B] (\ell) \not \in [V_{A}^{s}] \cup [V_{A}^{cu}]$
by Proposition \ref{pro:exit2}. Consider $\psi \in \tilde{G}$ such that
$j^{1} \psi = B$.
We deduce that $\lim_{k \to \infty} [(\psi \circ \phi^{p_{k}})(q)] = [B] (\ell)$.
Since $[B] (\ell) \not \in [V_{\phi}^{cu}]$, all accumulation points of the sequence
$([A^{-p}] ([B] \ell))_{p \geq 1}$ belong to $[V_{\phi}^{s}]$.
We deduce that there exists a sequence $(m_{k})_{k \geq 1}$ of natural numbers such
that $\lim_{k \to \infty} (\phi^{-m_{k}} \circ \psi \circ \phi^{p_{k}})(q) \in D^{s}$.
Moreover $(\psi \circ \phi^{p_{k}})(q)$ does not belong to $W_{\phi}^{s}$ for $k>>1$
since $[B] (\ell) \not \in [V_{\phi}^{s}]$.
Since $W_{\phi}^{s}$ is $\phi$-invariant, the point
$(\phi^{-m_{k}} \circ \psi \circ \phi^{p_{k}})(q)$ does not belong to
$W_{\phi}^{s}$ for any $k >>1$. In particular
$\lim_{k \to \infty} (\phi^{-m_{k}} \circ \psi \circ \phi^{p_{k}})(q)$
contains infinitely many points of the $G$-orbit of $q$ in every  of its neighborhoods.
Next lemma applied to $F=D^{s}$ implies that there exists a recurrent point for the action of the
pseudogroup ${\mathcal P}$ induced by $G$ where we suppose that $\phi$ is defined in
a neighborhood of $W_{\phi}^{s}$ and we can consider any domain of definition for
the elements of $G \setminus \langle \phi \rangle$.
\end{proof}
Given a set $S$ we denote by $\wp(S)$ the power set of $S$.
Let ${\mathcal P}$ be a pseudogroup of homeomorphisms defined in a topological space $M$.
Consider a closed subset $F$ of $M$. We consider the map
$\tau: F \to \wp(F)$ where $\tau (q)$ is defined as the set of points
$q' \in F$ such that $q'$ contains infinitely many points of the ${\mathcal P}$-orbit
of $q$ for any neighborhood of $q'$. The map extends naturally to a
map $\tau: \wp (F) \to \wp(F)$ by defining $\tau (S) = \cup_{q \in S} \tau (q)$ for any $S \in \wp (F)$.
\begin{lem}
\label{lem:funct}
Let ${\mathcal P}$ be a pseudogroup of homeomorphisms defined in a Hausdorff topological space $M$.
Consider a compact subset $F$ of $M$. Suppose that $\tau (q) \neq \emptyset$ for any
$q \in  F$. Then there exists a ${\mathcal P}$-recurrent point $q' \in F$ or in other
words $q' \in \tau (q')$.
\end{lem}
\begin{proof}
Let us enumerate some simple properties of $\tau$:
\begin{itemize}
\item $\tau (q)$ is a non-empty closed subset of $F$ for any $q \in F$.
\item $\tau (\tau (q)) \subset \tau (q)$ for any $q \in F$.
\end{itemize}
Let $S$ be the subset of $\wp (F)$ consisting of non-empty closed subsets $T$ of $F$
such that $\tau (T) \subset T$. Consider the order $(S, \supset)$ defined by the reverse inclusion.
Given a chain $C$ in $S$ the set $\cap_{T \in C} T$ is a closed subset
such that $T \supset \cap_{T' \in C} T'$ for any $T \in C$.
It is non-empty since the elements of the chain are compact.
We have $\tau (\cap_{T \in C} T) \subset \cap_{T \in C} \tau (T) \subset \cap_{T \in C} T$ and
then $\cap_{T \in C} T$ is an upper bound of the chain $C$.
Zorn's lemma implies that there exists a maximal element $T_{0}$ in $S$.
Notice that
$\tau (q)$ for $q \in T_{0}$ is a non-empty closed
subset of $T_{0}$ that belongs to $S$
by the second property above. Since $T_{0}$ is minimal for the inclusion,
we obtain $\tau (q) = T_{0}$ for any $q \in T_{0}$.
In particular $q$ belongs to $\tau (q)$ for any $q \in T_{0}$.
\end{proof}
The proof of Theorem \ref{teo:lcoivs} splits in two parts, namely
the cases that are consequence of Theorem \ref{teo:rec} and the
``hyperbolic" case. The proof in the later case has similarities with the
proof in dimension $2$ by Rebelo and Reis
\cite{RR:arxiv}; for instance Lemma \ref{lem:funct} is inspired in
one of their constructions.

\bibliography{rendu}
\end{document}